\documentclass[11pt]{amsart}

\usepackage{fancyhdr,color,graphicx,amsmath,amsfonts,amssymb,latexsym,bm,indentfirst,cite,amsthm,enumerate}
\usepackage[colorlinks=true,backref=page]{hyperref}

\newtheorem*{remark}{Remark}
\newtheorem*{claim}{Claim}
\newtheorem{Theorem}{Theorem}[section]
\newtheorem{Definition}[Theorem]{Definition}
\newtheorem{Lemma}[Theorem]{Lemma}

\newtheorem{Proposition}[Theorem]{Proposition}
\newtheorem{Example}[Theorem]{Example}
\newtheorem{fact}[Theorem]{Fact}

\numberwithin{equation}{section}
\newtheorem{theoremalph}{Theorem}

\begin{document}
   \newpage
   \title[Dimension theory of Diophantine approximation]{Dimension theory of Diophantine approximation related to $\beta$-transformations}
 
   \author{Wanlou Wu}
   \address[Wanlou Wu]{Department of Mathematics, Soochow University, Suzhou, 215006, China}
   \address{Laboratoire d'Analyse et de Math\'{e}matiques Appliqu\'{e}es, Universit\'{e} Paris-Est Cr\'{e}teil Val de Marne, Cr\'{e}teil, 94010, France}
   \email{wuwanlou@163.com}

   \author{Lixuan Zheng}
   \address[Lixuan Zheng]{Department of Mathematics, South China University of Technology, Guangzhou, 510000, China}
   \address{Laboratoire d'Analyse et de Math\'{e}matiques Appliqu\'{e}es, Universit\'{e} Paris-Est Cr\'{e}teil Val de Marne, Cr\'{e}teil, 94010, France}
   \email{jadezhenglixuan@gmail.com}
  
  

\begin{abstract}
   Let $T_\beta$ be the $\beta$-transformation on $[0,1)$ defined by $$T_\beta(x)=\beta x\text{ mod }1.$$ We study the Diophantine approximation of the orbit of a point $x$ under $T_\beta$. Precisely, for given two positive functions  $\psi_1,~\psi_2:\mathbb{N}\rightarrow\mathbb{R}^+$, define 
   $$\mathcal{L}(\psi_1):=\left\{x\in[0,1]:T_\beta^n x<\psi_1(n),\text{ for infinitely many $n\in\mathbb{N}$}\right\},$$ $$\mathcal{U}(\psi_2):=\left\{x\in [0,1]:\forall~N\gg1,~\exists~n\in[0,N],~s.t.~T^n_\beta x<\psi_2(N)\right\},$$ where $\gg$ means large enough. We compute the Hausdorff dimension of the set $\mathcal{L}(\psi_1)\cap\mathcal{U}(\psi_2)$. As a corollary, we estimate the Hausdorff dimension of the set $\mathcal{U}(\psi_2)$. 
\end{abstract}

\maketitle  	
	
\section{Introduction}
   Diophantine approximation, which originaly asks how closely can a given irrational number be approximated by a rational number $p/q$ with denominator $q$ no larger than a fixed positive integer $q_0$, has been widely studied by mathematicians. In 1842, Dirichlet \cite{D1842} proved the following theorem.\\
   {\bf Dirichlet Theorem} Given two real numbers $\theta,~Q$ with $Q\geq 1$, there is an integer $n$ with $1\leq n\leq Q$ such that $$\lVert n \theta\rVert<Q^{-1},$$ where $\lVert\xi\rVert$ denotes the distance from $\xi$ to the nearest integer. 
  
   Dirichlet Theorem is called a \emph{uniform approximation theorem} in \cite[pp.2]{W12}. A weak form of Dirichlet Theorem, called an \emph{asymptotic approximation theorem} in \cite[pp.2]{W12}, which was often refered to as a corollary of Dirichlet Theorem in the litterature has already existed in the book of Legendre \cite[1808, pp.18-19]{L2009} (using a continued fraction fact): for any real number $\theta$, there are infinitely many $n\in\mathbb{N}$ such that $$\lVert n\theta\rVert<n^{-1}.$$ For the general case, Khintchine in 1924 \cite{K1924} showed that for a positive function $\psi:\mathbb{N}\rightarrow\mathbb{R}^+$, if $x\mapsto x\psi(x)$ is non-increasing, then $$\mathcal{L}_\psi:=\left\{\theta\in\mathbb{R}:\lVert n\theta\rVert<\psi(n),\text{ for infinitely many }n\in\mathbb{N}\right\}$$ has Lebesgue measure zero if the series $\sum\psi(n)$ converges and has full Lebesgue measure otherwise. In the case where the set has Lebesgue measure zero, it is natural to calculate the Hausdorff dimension of $\mathcal{L}_\psi$. The first result on the Hausdorff dimension of $\mathcal{L}_\psi$ dates back to Jarn\'{\i}k-Bosicovitch Theorem \cite{B34,J29}. It was shown that the set $$\left\{\theta\in\mathbb{R}:\lVert n\theta\rVert<\dfrac{1}{n^\tau},\text{ for infinitely many }n\in\mathbb{N}\right\}$$ has Hausdorff deminsion $\dfrac{2}{1+\tau}$, for any $\tau>1$.       
   
   In analogy with the classical Diophantine approximation, Hill and Velani \cite{HV1995} studied the approximation properties of the orbits of a dynamical system and introduced the so called \emph{shrinking target problems}: for a measure preserving transformation $T:M\rightarrow M$ on a manifold $M$, what is the size (Lebesgue measure, Hausdorff dimension) of the set $$\left\{x\in M:T^nx\in B(n),\text{ for infinitely many $n\in\mathbb{N}$}\right\},$$ where $B(n)=B(x_0, r(n))$ is a ball centred at $x_0$ with radius $r(n)$($r(n)\rightarrow 0$)? They answered the case where $T$ is an expanding rational map of the Riemann sphere $\overline{\mathbb{C}}=\mathbb{C}\cup\{\infty\}$. 
   
   In this papper, we are interested in the approximation properties of the orbits of $\beta$-transformations. The $\beta$-transformation $T_\beta~(\beta>1)$ on $[0,1)$ is defined by $$T_\beta(x):=\beta x-\lfloor\beta x\rfloor,$$ where $\lfloor\cdot\rfloor$ is the integer part function. For any positive function  $\psi:\mathbb{N}\rightarrow\mathbb{R}^+$, define the set of \emph{$\psi$-well asymptotically approximable} points by $x_0$ as $$\mathcal{L}(\psi,x_0):=\{x\in[0,1]:\lvert T_\beta^n x-x_0\rvert<\psi(n),\text{ for infinitely many $n\in\mathbb{N}$}\}.$$ By \cite[Theorem 2A, B, C]{P67}, the set $\mathcal{L}(\psi,x_0)$ has Lebesgue measure zero if and only if the series $\sum\psi(n)$ converges. Shen and Wang \cite[Theorem 1.1]{SW2013} established the following result on the Hausdorff dimension of $\mathcal{L}(\psi,x_0)$. \\ \emph{{\bf Theorem SW}(\cite[Theorem 1.1]{SW2013}) For any real number $\beta>1$ and any point $x_0\in[0,1]$, one has $${\rm dim}_H \left(\mathcal{L}(\psi, x_0)\right)=\dfrac{1}{1+v},\quad\text{where $v:=\liminf\limits_{n\rightarrow\infty}\dfrac{-\log_\beta\psi(n)}{n}$}.$$}
   
   Parallel to the asymptotic approximation theorem, it is also worth of studying the uniform approximation properties as in Dirichlet Theorem. The uniform Diophantine approximation related to $\beta$-transformations was studied by Bugeaud and Liao \cite{YLiao2016}. For $x\in[0,1)$, let $$\nu_\beta(x):=\sup\left\{v\geq 0:T^n_\beta x<(\beta^n)^{-v},\text{ for infinitely many $n\in\mathbb{N}$}\right\},$$ $$\hat{\nu}_\beta(x):=\sup\left\{v\geq 0:\forall~N\gg 1,~T^n_\beta x<(\beta^N)^{-v}\text{ has a solution $n\in[0,N]$}\right\}.$$ Bugeaud and Liao \cite{YLiao2016} proved the following theorem.\\ \emph{{\bf Theorem BL} (\cite[Theorem 1.4]{YLiao2016}) For any $v\in(0,+\infty)$ and any $\hat{v}\in(0,1)$, if $v<\hat{v}/(1-\hat{v})$, then the set $$\left\{x\in[0,1]:\nu_\beta(x)= v\right\}\cap\left\{x\in[0,1]:\hat{\nu}_\beta(x)\geq \hat{v}\right\}$$ is empty. Otherwise, $${\rm dim}_H \left(\{x\in[0,1]:\nu_\beta(x)= v\}\cap\{x\in[0,1]:\hat{\nu}_\beta(x)=\hat{v}\}\right)=\dfrac{v-\hat{v}-v\hat{v}}{(1+v)(v-\hat{v})}.$$}The exponents $\nu_\beta$ and $\hat{\nu}_\beta$ were introduced in \cite{AB10}(see also \cite[Ch.7]{B2012}). They are strongly related to the run-length function of $\beta$-expansions (see \cite{Z18}). The aim of this paper is to study the Diophantine approximation sets in \cite{YLiao2016} when the approximation speed function $n\mapsto\beta^{-nv}$ is replaced by a general positive function. More precisely, fix two positive functions $\psi_1,~\psi_2:\mathbb{N}\rightarrow\mathbb{R}^+$, and define  $$\mathcal{L}(\psi_1):=\left\{x\in [0,1]:T_\beta^n x<\psi_1(n),\text{ for infinitely many $n\in\mathbb{N}$}\right\},$$ $$\mathcal{U}(\psi_2):=\left\{x\in [0,1]:\forall~N\gg 1,~T^n_\beta x<\psi_2(N)\text{ has a solution $n\in[0,N]$}\right\}.$$ We will estimate the Hausdorff dimension of the sets $\mathcal{L}(\psi_1)\cap\mathcal{U}(\psi_2)$ and $\mathcal{U}(\psi_2)$. Let $$\underline{v}_1:=\liminf_{n\rightarrow\infty}\dfrac{-\log_\beta\psi_1(n)}{n},\qquad\overline{v}_1:=\limsup_{n\rightarrow\infty}\dfrac{-\log_\beta\psi_1(n)}{n};$$ $$ \underline{v}_2:=\liminf_{n\rightarrow\infty}\dfrac{-\log_\beta\psi_2(n)}{n},\qquad\overline{v}_2:=\limsup_{n\rightarrow\infty}\dfrac{-\log_\beta\psi_2(n)}{n}.$$ If $\underline{v}_1<0$, by the definition of $\underline{v}_1$, there is a sequence $\{n_j\}$ such that $$\lim_{j\rightarrow\infty}\dfrac{-\log_\beta\psi_1(n_j)}{n_j}=\underline{v}_1.$$ Then, for $\varepsilon>0$ small enough, there exists an integer $j_0$ such that $$1<\beta^{-n_j(\underline{v}_1+\varepsilon)}\leq\psi_1(n_j),~\text{for any $j\geq j_0$}.$$ By the fact $T^n_\beta x<1$, for any $x\in[0,1)$ and any $n\in\mathbb{N}$, for any $x\in[0,1)$, we have $T^{n_j}_\beta x<1<\psi_1(n_j)$. This implies $$[0,1)\subseteq\mathcal{L}(\psi_1).$$ On the other hand, if we take all the integers $n_i$ with the following property $$\psi_2(n_i)>1,~\text{for $i=1,2,3\cdots$}$$ then for any $x\in[0,1)$ and any integer $n\in[1,n_i]$, we have $T^n_\beta x<1<\psi_2(n_i)$. Thus, we can replace $\psi_2(n)$ by the function    
\begin{equation*}
   \widetilde{\psi}_2(n)=
   \begin{cases}
   \psi_2(n), &\text{ if $n\neq n_i$}\\
   1, &\text{ if $n=n_i$}
   \end{cases},~i=1,2,\cdots
\end{equation*} 
   The size (Lebesgue measure, Hausdorff dimension) of the sets $\mathcal{L}(\psi_1)\cap\mathcal{U}(\psi_2)$ and $\mathcal{U}(\psi_2)$ are the same as that of the sets $\mathcal{L}(\psi_1)\cap\mathcal{U}(\widetilde{\psi}_2)$ and $\mathcal{U}(\widetilde{\psi}_2)$. Therefore, in this paper, we always assume $\underline{v}_1\geq0$ and $\underline{v}_2\geq0$. We establish the following theorems on the Hausdorff dimension of the set $\mathcal{L}(\psi_1)\cap\mathcal{U}(\psi_2)$.     
\begin{theoremalph}\label{Special-case}
\begin{enumerate}[$(1)$]
   \item If $\underline{v}_1=\overline{v}_1=\underline{v}_2=\overline{v}_2=0$,  then $${\rm dim}_H \left(\mathcal{L}(\psi_1)\cap\mathcal{U}(\psi_2)\right)=1;$$
   \item If $\underline{v}_2=\infty$ and $0\leq\underline{v}_1\leq\overline{v}_1\leq\infty$, then $\mathcal{L}(\psi_1)\cap\mathcal{U}(\psi_2)$ is countable. 
   \item If $\underline{v}_1=\infty$ and $0\leq\underline{v}_2\leq\overline{v}_2\leq\infty$, then ${\rm dim}_H \left(\mathcal{L}(\psi_1)\cap\mathcal{U}(\psi_2)\right)=0$. 
\end{enumerate}	  
\end{theoremalph}

\begin{remark}
	For Item (1), the set $\mathcal{L}(\psi_1)\cap\mathcal{U}(\psi_2)$ is not necessary of full Lebesgue measure. In fact, if the series $\sum\psi_1(n)$ converges, by \cite[Theorem 2A, B, C]{P67}, $$m\left(\mathcal{L}(\psi_1)\cap\mathcal{U}(\psi_2)\right)=0,$$ where $m(A)$ denotes the Lebesgue measure of $A$. However, the set $\mathcal{L}(\psi_1)\cap\mathcal{U}(\psi_2)$ can also be of full Lebesgue measure. For example, if $\psi_1(n)=\psi_2(n)=1/n$, according to Dmitry, Konstantoulas, and Florian \cite[Theorem 1.1]{KKR19}, $$m\left(\mathcal{L}(\psi_1)\cap\mathcal{U}(\psi_2)\right)=1.$$ 
	
	For Item $(3)$, if $\underline{v}_2=\infty$, then $\mathcal{L}(\psi_1)\cap\mathcal{U}(\psi_2)$ is countable. If $1<\underline{v}_2<\infty$, then $\mathcal{L}(\psi_1)\cap\mathcal{U}(\psi_2)$ is empty (see Lemma \ref{empty}). If $0<\underline{v}_2\leq1$, then $\mathcal{L}(\psi_1)\cap\mathcal{U}(\psi_2)$ is uncountable (see Proposition \ref{uncountable}).
\end{remark}
     
\begin{theoremalph}\label{Nonspecial}
   If $\underline{v}_2>1$, then $\mathcal{L}(\psi_1)\cap\mathcal{U}(\psi_2)$ is countable. If $\overline{v}_1/(2+\overline{v}_1)\leq\underline{v}_2\leq1<\overline{v}_2$, then $$0\leq{\rm dim}_H \left(\mathcal{L}(\psi_1)\cap\mathcal{U}(\psi_2)\right)\leq\min\left\{\dfrac{1}{1+\overline{v}_2},~\left(\dfrac{1-\underline{v}_2}{1+\underline{v}_2}\right)^2\right\}.$$ If $\underline{v}_1/(2+\underline{v}_1)<\underline{v}_2\leq\overline{v}_2\leq1$ and $\overline{v}_1/(2+\overline{v}_1)<\overline{v}_2$, then $$\left(\dfrac{1-\overline{v}_2}{1+\overline{v}_2}\right)^2\leq{\rm dim}_H \left(\mathcal{L}(\psi_1)\cap\mathcal{U}(\psi_2)\right)\leq\min\left\{\dfrac{1}{1+\overline{v}_2},~ \left(\dfrac{1-\underline{v}_2}{1+\underline{v}_2}\right)^2\right\}.$$ If $\underline{v}_2\leq\underline{v}_1/(2+\underline{v}_1)$ and $\overline{v}_1/(2+\overline{v}_1)<\overline{v}_2\leq1$, then $$\left(\dfrac{1-\overline{v}_2}{1+\overline{v}_2}\right)^2\leq{\rm dim}_H \left(\mathcal{L}(\psi_1)\cap\mathcal{U}(\psi_2)\right)\leq\dfrac{\underline{v}_1-\underline{v}_2-\underline{v}_1\cdot\underline{v}_2}{(1+\underline{v}_1)(\underline{v}_1-\underline{v}_2)}.$$ If $\underline{v}_1/(2+\underline{v}_1)<\underline{v}_2\leq\overline{v}_2\leq\overline{v}_1/(2+\overline{v}_1)$, then $$\dfrac{\overline{v}_1-\overline{v}_2-\overline{v}_1\cdot\overline{v}_2}{(1+\overline{v}_1)(\overline{v}_1-\overline{v}_2)}\leq{\rm dim}_H\left(\mathcal{L}(\psi_1)\cap\mathcal{U}(\psi_2)\right)\leq\min\left\{\dfrac{1}{1+\overline{v}_2},~\left(\dfrac{1-\underline{v}_2}{1+\underline{v}_2}\right)^2\right\}.$$ If $\underline{v}_2\leq\underline{v}_1/(2+\underline{v}_1)$ and $\overline{v}_2\leq\overline{v}_1/(2+\overline{v}_1)$, then $$\dfrac{\overline{v}_1-\overline{v}_2-\overline{v}_1\cdot\overline{v}_2}{(1+\overline{v}_1)(\overline{v}_1-\overline{v}_2)}\leq{\rm dim}_H\left(\mathcal{L}(\psi_1)\cap\mathcal{U}(\psi_2)\right)\leq\dfrac{\underline{v}_1-\underline{v}_2-\underline{v}_1\cdot\underline{v}_2}{(1+\underline{v}_1)(\underline{v}_1-\underline{v}_2)}.$$
\end{theoremalph}   

   We remark that Theorems \ref{Special-case} and \ref{Nonspecial} give all the cases. We also estimate the Hausdorff dimension of $\mathcal{U}(\psi_2)$.
\begin{theoremalph}\label{Uniform}
   If $\underline{v}_2>1$, then $\mathcal{U}(\psi_2)$ is countable. If $\underline{v}_2\leq1<\overline{v}_2$, then $$0\leq{\rm dim}_H \left(\mathcal{U}(\psi_2)\right)\leq\min\left\{\dfrac{1}{1+\overline{v}_2},~\left(\dfrac{1-\underline{v}_2}{1+\underline{v}_2}\right)^2\right\}.$$ If $\overline{v}_2\leq1$, then $$\left(\dfrac{1-\overline{v}_2}{1+\overline{v}_2}\right)^2\leq{\rm dim}_H \left(\mathcal{U}(\psi_2)\right)\leq\min\left\{\dfrac{1}{1+\overline{v}_2},~\left(\dfrac{1-\underline{v}_2}{1+\underline{v}_2}\right)^2\right\}.$$ 
\end{theoremalph}

   We will show in Examples \ref{Exa1}, \ref{Exa2}, \ref{Exa3}, \ref{Exa4}, \ref{Exa5}, \ref{Exa6} and \ref{Exa7}, that the upper and lower bound of the Hausdorff dimension in Theorems \ref{Nonspecial} and \ref{Uniform} can be all reached. When $\underline{v}_1=\overline{v}_1=0$, we have the result as Theorem \ref{Belong}.
\begin{theoremalph}\label{Belong}
   Assume $\underline{v}_1=\overline{v}_1=0$. If $\underline{v}_2>0$, then $\mathcal{U}(\psi_2)\subseteq \mathcal{L}(\psi_1)$. 
\end{theoremalph}
   
   Our paper is organized as follows. We recall some classical results of the theory of $\beta$-expansion in Section $2$. Theorems \ref{Special-case} and \ref{Nonspecial} are proved in Section $3$. Section $4$ establishes  Theorems \ref{Uniform} and \ref{Belong}. In Section $5$, we give examples to show that the estimations in Theorems \ref{Nonspecial} and \ref{Uniform} are sharp.  
   
\medskip  
\section{$\beta$-expansions}
   The notion of $\beta$-expansion was introduced by R\'{e}nyi \cite{Ren57} in 1957. For any $\beta>1$, the $\beta$-transformation $T_{\beta}$ on $[0,1)$ is defined by $$T_\beta x=\beta x-\lfloor\beta x\rfloor,$$ where $\lfloor\xi\rfloor$ denotes the largest integer less than or equal to $\xi$. Let
\begin{equation*} \lceil\beta\rfloor=
  \begin{cases}
  \beta -1, &  \text{if}~\beta \text{ is a positive integer},\\
  \lfloor \beta \rfloor, &  \text{otherwise}.
  \end{cases}
\end{equation*}
  
\begin{Definition}
   The $\beta$-expansion of a number $x\in[0,1)$ is the sequence $\{\varepsilon_n\}_{n\geq 1}:=\{\varepsilon_n(x,\beta)\}_{n\geq 1}$ of integers from $\{0,1,\cdots,\lceil\beta\rfloor\}$ such that
\begin{equation}\label{E1}
   x=\dfrac{\varepsilon_1}{\beta}+\dfrac{\varepsilon_2}{\beta^2}+\cdots+\dfrac{\varepsilon_n}{\beta^n}+\cdots,
\end{equation}
   where $$\varepsilon_1=\varepsilon_1(x,\beta)=\lfloor\beta x \rfloor,~\varepsilon_n=\varepsilon_n(x,\beta)=\lfloor\beta T^{n-1}_\beta x\rfloor,~\text{ for all $n\geq 2$}.$$ We also write $ d_\beta(x)=\left(\varepsilon_1,\cdots,\varepsilon_n,\cdots\right)$.
\end{Definition}
  
   We can extend the definition of the $\beta$-transformation to the point $1$ as:$$T_\beta 1=\beta-\lfloor\beta \rfloor.$$ One can obtain $$ 1=\dfrac{\varepsilon_1(1,\beta)}{\beta}+\dfrac{\varepsilon_2(1,\beta)}{\beta^2}+\cdots+\dfrac{\varepsilon_n(1,\beta)}{\beta^n}+\cdots,$$ where $\varepsilon_1(1,\beta)=\lfloor\beta\rfloor,~\varepsilon_n=\lfloor\beta T^{n-1}_\beta 1\rfloor,~\text{for all $n\geq 2$}$. We also write $$ d_\beta(1)=\left(\varepsilon_1(1,\beta),\cdots,\varepsilon_n(1,\beta),\cdots\right).$$ If $d_\beta(1)$ is finite, i.e., there is an integer $m>0$ such that $\varepsilon_m(1,\beta)\neq 0$ and $\varepsilon_i(1,\beta)=0$ for all $i>m$, then $\beta$ is called a \emph{simple Parry number}. In this case, the infinite $\beta$-expansion of $1$ is defined as: $$(\varepsilon^\ast_1(\beta),\varepsilon^\ast_2(\beta),\cdots,\varepsilon^\ast_n(\beta),\cdots):=(\varepsilon_1(1,\beta),\varepsilon_2(1,\beta),\cdots,\varepsilon_m(1,\beta)-1)^\infty,$$ where $(\omega)^\infty$ denotes the periodic sequence $(\omega, \omega,\cdots)$. If $d_\beta(1)$ is infinite, then we define $$(\varepsilon^\ast_1(\beta),\varepsilon^\ast_2(\beta),\cdots,\varepsilon^\ast_n(\beta),\cdots):=(\varepsilon_1(1,\beta),\varepsilon_2(1,\beta),\cdots,\varepsilon_n(1,\beta),\cdots).$$
   
   Endow the set $\{0,1,\cdots,\lceil\beta\rfloor\}^{\mathbb{N}}$ with the product topology and define the one-sided shift operator $\sigma$ as:$$\sigma\left((\omega_n)_{n\geq1}\right):=(\omega_{n+1})_{n\geq1},$$ for any infinite sequence $(\omega_n)_{n\geq1}$ in $\{0,1,\cdots,\lceil\beta\rfloor\}^{\mathbb{N}}$. The lexicographical order $<_{lex}$ on $\{0,1,\cdots,\lceil\beta\rfloor\}^{\mathbb{N}}$ is defined as: $$\omega=(\omega_1,\omega_2,\cdots)<_{lex}\omega'=(\omega'_1,\omega'_2,\cdots),$$ if $\omega_1<\omega'_1$ or if there is an integer $k\geq 2$ such that for all $1\leq i< k$, $\omega_i=\omega'_i$ but $\omega_k<\omega'_k$. Denote by $\omega\leq_{lex}\omega'$ if $\omega<_{lex}\omega'$ or $\omega=\omega'$.
  
\begin{Definition}
   A finite word $(\omega_1,\omega_2,\cdots,\omega_n)$ is called $\beta$-admissible, if there is $x\in[0,1]$ such that the $\beta$-expansion of $x$ begins with $(\omega_1,\omega_2,\cdots,\omega_n).$ An infinite sequence $(\omega_1,\omega_2,\cdots,\omega_n,\cdots)$ is called $\beta$-admissible, if there is $x\in[0,1]$ such that the $\beta$-expansion of $x$ is $(\omega_1,\omega_2,\cdots,\omega_n,\cdots).$
\end{Definition} 
  
   Denote by $\Sigma_\beta$ the set of all infinite $\beta$-admissible sequences and $\Sigma^n_\beta$ the set of all $\beta$-admissible sequences with length $n$. The $\beta$-admissible sequences are characterized by Parry \cite{P1960} and R\'{e}nyi \cite{Ren57}.
\begin{Theorem}\label{Admissible}
   Let $\beta>1$,
\begin{enumerate}[(1)]
  \item(\cite[Lemma 1]{P1960})
   A word $\omega=(\omega_n)_{n\geq 1}\in\Sigma_\beta$ if and only if $$\sigma^k(\omega)\leq_{lex}(\varepsilon^\ast_1(\beta),\varepsilon^\ast_2(\beta),\cdots,\varepsilon^\ast_n(\beta),\cdots),~\text{for all $k\geq 0$}.$$ 
  	
  \item(\cite[Lemma 3]{P1960}) For any $x_1,~x_2\in[0,1]$, $x_1<x_2$ if and only if $$d_\beta(x_1)<_{lex}d_\beta(x_2).$$
  	
  \item(\cite[Lemma 4]{P1960}) For any $\beta_2>\beta_1>1$, one has $$\Sigma^n_{\beta_1}\subseteq\Sigma^n_{\beta_2},\quad\Sigma_{\beta_1}\subseteq\Sigma_{\beta_2}.$$
\end{enumerate}	
\end{Theorem} 

\begin{Theorem}\label{cardinality}(\cite[Theorem 2]{Ren57})
   For any $\beta>1$, one has $$\beta^n\leq\sharp \Sigma^n_\beta\leq\dfrac{\beta^{n+1}}{\beta-1},$$ where $\sharp$ denotes the cardinality of a finite set.
\end{Theorem} 

   For every $(\omega_1,\cdots,\omega_n)\in\Sigma^n_\beta$, we call $$I_n(\omega_1,\cdots,\omega_n):=\{x\in[0,1]:d_\beta(x)\text{ starts with }(\omega_1,\cdots,\omega_n)\}$$ an \emph{$n$-th order basic interval} with respect to $\beta$. Denote by $I_n(x)$ the $n$-th order basic interval containing $x$. The basic intervals are also called \emph{cylinders} by some authors. It is crucial to estimate the lengths of the basic intervals. We will use the key notion of \textquotedblleft full basic interval\textquotedblright as follows (see \cite{LW2008, FW2012}).

\begin{Definition}
   For any $(\omega_1,\cdots,\omega_n)\in\Sigma^n_\beta$, a basic interval $I_n(\omega_1,\cdots,\omega_n)$ is said to be full if its length is $\beta^{-n}$.
\end{Definition}  

\begin{Proposition}\label{full}
   (\cite[Lemma 3.1]{FW2012} and \cite[Lemma 2.5]{SW2013})
	 
   For any $(\omega_1,\cdots,\omega_n)\in\Sigma^n_\beta$, the following statements are equivalent:
\begin{enumerate}[(1)]
   \item $I_n(\omega_1,\cdots,\omega_n)$ is a full basic interval;
   \item $T^n_\beta I_n(\omega_1,\cdots,\omega_n)=[0,1)$;
   \item For any $\omega'=(\omega'_1,\cdots,\omega'_m)\in\Sigma^m_\beta$, the concatenation $$(\omega_1,\cdots,\omega_n,\omega'_1,\cdots,\omega'_m)\in\Sigma^{n+m}_\beta, \text{ i.e., is $\beta$-admissible.}$$ 
\end{enumerate}
\end{Proposition}

\begin{Proposition}\label{fullc}
	(\cite[Corollary 2.6]{SW2013})
\begin{enumerate}[(1)]
   \item If $(\omega_1,\cdots,\omega_{n+1})$ is a $\beta$-admissible sequence with $\omega_{n+1}\neq0$, then $$I_{n+1}(\omega_1,\cdots,\omega'_{n+1})$$ is full for any $0\leq \omega'_{n+1}<\omega_{n+1}$. 
   \item For every $\omega\in\Sigma^n_\beta$, if $I_n(\omega)$ is full, then for any $\omega'\in\Sigma^m_\beta$, one has $$\lvert I_{n+m}(\omega,\omega')\rvert=\lvert I_n(\omega)\rvert\cdot\lvert I_m(\omega')\rvert=\dfrac{\lvert I_m(\omega')\rvert}{\beta^n}.$$
   \item For any $\omega\in\Sigma^n_\beta$, if $I_{n+m}(\omega,\omega')$ is a full basic interval contained in $I_n(\omega)$ with the smallest order, then $$\lvert I_{n+m}(\omega,\omega')\rvert\geq\dfrac{\lvert I_n(\omega)\rvert}{\beta}.$$ 
\end{enumerate}
\end{Proposition}

   Next, we define a sequence of numbers $\beta_N$ approaching to $\beta$ as follows. Let $\{\varepsilon^\ast_k(\beta):k\geq 1\}$ be the infinite $\beta$-expansion of $1$. Let $\beta_N$ be the unique real solution of the equation
\begin{equation}\label{ED1}
   1=\dfrac{\varepsilon^\ast_1(\beta)}{z}+\cdots+\dfrac{\varepsilon^\ast_N(\beta)}{z^N}.
\end{equation} 
   Therefore, $\beta_N<\beta$ and the sequence $\{\beta_N:N\geq 1\}$ increases and converges to $\beta$ when $N$ tends to infinity.

\begin{Lemma}\label{length}
   (\cite[Lemma 2.7]{SW2013}) For every $\omega\in\Sigma^n_{\beta_N}$ viewed as an element of $\Sigma^n_\beta$, one has $$\dfrac{1}{\beta^{n+N}}\leq\lvert I_n(\omega_1,\cdots,\omega_n)\rvert\leq\dfrac{1}{\beta^n}.$$
\end{Lemma}

\medskip

\section{Proofs of Theorems \ref{Special-case} and \ref{Nonspecial} }\label{sec2}
   First, we give an easy fact which is useful for the proofs of Theorems \ref{Special-case} and \ref{Nonspecial}.
\begin{fact}\label{fact}
	For any $x\in[0,1)$, if there is an integer $n_0$ such that $T^{n_0}_\beta x=0$, then $$\nu_\beta(x)=\hat{\nu}_\beta(x)=\infty.$$
\end{fact}   

   Now, we prove that if $\overline{v}_2=\infty$, then the Hausdorff dimensions of the sets $\mathcal{L}(\psi_1)\cap\mathcal{U}(\psi_2)$ and $\mathcal{U}(\psi_2)$ are zero.   
\begin{Proposition}\label{Pro}
   If $\overline{v}_2=\infty$, then $\mathcal{U}(\psi_2)\subseteq\{x\in[0,1]:\nu_\beta(x)=\infty\}$. Thus, $${\rm dim}_H(\mathcal{L}(\psi_1)\cap\mathcal{U}(\psi_2))={\rm dim}_H(\mathcal{U}(\psi_2))=0.$$	
\end{Proposition}

\begin{proof}
   For every $x\in\mathcal{U}(\psi_2)$, we distinguish two cases:
	
   {\bf Case 1:} There is an integer $n_0$ such that $T^{n_0}_\beta x=0$. By Fact \ref{fact}, we have $$\nu_\beta(x)=\infty.$$ 
	
   {\bf Case 2:} For any $n\in\mathbb{N}$, we always have $T^n_\beta x>0$. Since $x\in\mathcal{U}(\psi_2)$, there is $N_0\geq1$ such that for any $N\geq N_0$, there is an integer $n\in[0,N]$ such that $$0<T_\beta^nx<\psi_2(N).$$ Since $\limsup\limits_{n\rightarrow\infty}\dfrac{-\log_\beta\psi_2(n)}{n}=\overline{v}_2=\infty$, for any $L>0$ large enough, there is a sequence $\{n_i\}$ such that $\psi_2(n_i)\leq\beta^{-n_iL}$. Let $m_1:=\min\{n_i:~n_i\geq N_0\}$, there is an integer $j_1\in[0,m_1]$ such that $$0<T_\beta^{j_1}x<\psi_2(m_1)\leq\beta^{-m_1L}\leq\beta^{-j_1L}.$$ Take $m_2:=\min\left\{n_i>m_1:\beta^{-n_iL}<T_\beta^{j_1}x\right\}$. There is $j_2\in[0,m_2]$ such that  $$0<T_\beta^{j_2}x<\psi_2(m_2)\leq\beta^{-m_2L}\leq\beta^{-j_2L}.$$ Since $T_\beta^{j_2}x<\psi_2(m_2)\leq\beta^{-m_2L}<T_\beta^{j_1}x$, $j_2\neq j_1$. Repeat this process, one can get a sequence of pairwise disjoint integers $\{j_i:i\geq1\}$ such that $$0<T_\beta^{j_i}x<\psi_2(m_i)\leq\beta^{-m_iL}\leq\beta^{-j_iL}.$$ Therefore, $\nu_\beta(x)\geq L$. By the arbitrariness of $L$, we have $\nu_\beta(x)=\infty$. 
	
   Hence, in all cases, we have $$\mathcal{U}(\psi_2)\subseteq\{x\in[0,1]:\nu_\beta(x)=\infty\}.$$ By {\bf Theorem SW}, $${\rm dim}_H(\mathcal{L}(\psi_1)\cap\mathcal{U}(\psi_2))\leq{\rm dim}_H(\mathcal{U}(\psi_2))\leq{\rm dim}_H(\{x\in[0,1]:\nu_\beta(x)=\infty\})=0.$$
\end{proof}

   We discuss the relation between $\nu_\beta(x)$, $\hat{\nu}_\beta(x)$ and $\underline{v}_1$, $\overline{v}_1$, $\underline{v}_2$, $\overline{v}_2$, which are important to the proof of Theorem \ref{Nonspecial}. 
\begin{Lemma}\label{SET}
   For $0\leq\underline{v}_1\leq\overline{v}_1<\infty$ and $0\leq\underline{v}_2\leq\overline{v}_2<\infty$, one has
\begin{enumerate}[$(1)$]
   \item $\{x\in[0,1]:\nu_\beta(x)>\overline{v}_1\}\cap\{x\in[0,1]:\hat{\nu}_\beta(x)>\overline{v}_2\}\subseteq\mathcal{L}(\psi_1)\cap\mathcal{U}(\psi_2);$ 
   \item  $\mathcal{L}(\psi_1)\cap\mathcal{U}(\psi_2)\subseteq\{x\in[0,1]:\nu_\beta(x)\geq \underline{v}_1\}\cap\{x\in[0,1]:\hat{\nu}_\beta(x)\geq\underline{v}_2\}.$ 
 \end{enumerate}	
\end{Lemma}      
   
\begin{proof}
   $(1)$ For any $x$ with $\nu_\beta(x)>\overline{v}_1$ and any $\varepsilon>0$ small enough, there is a sequence $\{n_i\}$ such that $T_\beta^{n_i}x<\beta^{-n_i(\overline{v}_1+\varepsilon)}$. By the definition of $\overline{v}_1$, for the above $\varepsilon$, there is an integer $i_0$ such that $$\psi_1(n_i)>\beta^{-n_i(\overline{v}_1+\varepsilon)},\quad\text{for any $i\geq i_0$}.$$ Then, $$T_\beta^{n_i}x<\beta^{-n_i(\overline{v}_1+\varepsilon)}<\psi_1(n_i),~\text{for any $i\geq i_0$}.$$ Thus, $x\in\mathcal{L}(\psi_1)$. Therefore, $$\{x\in[0,1]:\nu_\beta(x)>\overline{v}_1\}\subseteq\mathcal{L}(\psi_1).$$ By similar discussion,  $\{x\in[0,1]:\hat{\nu}_\beta(x)>\overline{v}_2\}\subseteq \mathcal{U}(\psi_2).$ Thus, $$\{x\in[0,1]:\nu_\beta(x)>\overline{v}_1\}\cap\{x\in[0,1]:\hat{\nu}_\beta(x)>\overline{v}_2\}\subseteq\mathcal{L}(\psi_1)\cap\mathcal{U}(\psi_2).$$
  
   $(2)$ For any $x\in\mathcal{L}(\psi_1)$, there is a sequence $\{n_i\}$ such that $$T_\beta^{n_i}x<\psi_1(n_i).$$ By the definition of $\underline{v}_1$, for any $\varepsilon>0$, there is an integer $i_0$ such that $$T_\beta^{n_i}x<\psi_1(n_i)<\beta^{-n_i(\underline{v}_1-\varepsilon)},~\text{for any $i\geq i_0$}.$$ Thus, $\nu_\beta(x)\geq\underline{v}_1-\varepsilon$. Therefore, $$\mathcal{L}(\psi_1)\subseteq\{x\in[0,1]:\nu_\beta(x)\geq \underline{v}_1-\varepsilon\}.$$ By the arbitrariness of $\varepsilon$, one can obtain $$\mathcal{L}(\psi_1)\subseteq\cap_{\varepsilon>0}\{x\in[0,1]:\nu_\beta(x)\geq\underline{v}_1-\varepsilon\}=\{x\in[0,1]:\nu_\beta(x)\geq\underline{v}_1\}.$$ By similar discussion,  $\mathcal{U}(\psi_2)\subseteq\{x\in[0,1]:\hat{\nu}_\beta(x)\geq\underline{v}_2\}.$ Thus, $$\mathcal{L}(\psi_1)\cap\mathcal{U}(\psi_2)\subseteq\{x\in[0,1]:\nu_\beta(x)\geq\underline{v}_1\}\cap\{x\in[0,1]:\hat{\nu}_\beta(x)\geq\underline{v}_2\}.$$      
\end{proof}
   
   To prove Theorem \ref{Special-case}, we characterize the set of all points $x$ with $\nu_\beta(x)=\infty$ and $\hat{\nu}_\beta(x)=\infty$.  
\begin{Lemma}\label{infinite}
   If $\nu_\beta(x)=\infty$ and $\hat{\nu}_\beta(x)=\infty$, then one has $$(1)~\cup_{n=1}^{\infty}\cup_{\omega\in\Sigma^n_\beta}\{x\in[0,1]:d_\beta(x)=(\omega,0^\infty)\}\subseteq\{x\in[0,1]:\nu_\beta(x)=\infty\};$$ $$(2)~\cup_{n=1}^{\infty}\cup_{\omega\in\Sigma^n_\beta}\{x\in[0,1]:d_\beta(x)=(\omega,0^\infty)\}=\{x\in[0,1]:\hat{\nu}_\beta(x)=\infty\}.$$
\end{Lemma}	 

\begin{proof}
   By Fact \ref{fact}, the statement (1) and the inclusion $$\cup_{n=1}^{\infty}\cup_{\omega\in\Sigma^n_\beta}\{x\in[0,1]:d_\beta(x)=(\omega,0^\infty)\}\subseteq\{x\in[0,1]:\hat{\nu}_\beta(x)=\infty\}$$ are obvious. What is left is to show $$\{x\in[0,1]:\hat{\nu}_\beta(x)=\infty\}\subseteq\cup_{n=1}^{\infty}\cup_{\omega\in\Sigma^n_\beta}\{x\in[0,1]:d_\beta(x)=(\omega,0^\infty)\}.$$ By contrary, for any $x$ with $\hat{\nu}_\beta(x)=\infty$, we suppose $$x\notin\cup_{n=1}^{\infty}\cup_{\omega\in\Sigma^n_\beta}\{x\in[0,1]:d_\beta(x)=(\omega,0^\infty)\}.$$ Then $T^n_\beta x>0$ for every $n\in\mathbb{N}$. Denote the $\beta$-expansion of $x$ by $$x=\dfrac{a_1}{\beta}+\dfrac{a_2}{\beta^2}+\cdots+\dfrac{a_n}{\beta^n}+\cdots,$$ where $a_i\in\{0,\cdots,\lceil\beta\rfloor\}$, for all $i\geq1$. We can take two increasing sequences $\left\{n'_i:i\geq 1\right\}$ and $\left\{m'_i:i\geq1\right\}$ with the following properties:
\begin{enumerate}[(1)]
   \item For every $i\geq1$, one has $$a_{n'_i}>0,\quad a_{n'_i+1}=\cdots=a_{m'_i-1}=0,\quad a_{m'_i}>0.$$
   \item For every $a_n=0$, there is an integer $i$ such that $n'_i<n<m'_i$. 
\end{enumerate}
   By the choices of $\left\{n'_i: i\geq 1\right\}$ and $\left\{m'_i:i\geq1\right\}$, for every $i\geq1$, one has $n'_i<m'_i<n'_{i+1}$. Since $\hat{\nu}_\beta(x)>0$, one has $$\limsup_{i\rightarrow\infty}(m'_i-n'_i)=\infty.$$ Take $n_1=n'_1$ and $m_1=m'_1$. Suppose that $m_k,~n_k$ have been defined. Let $i_1=1$ and $i_{k+1}:=\min\{i>i_k:m'_i-n'_i> m_k-n_k\},~{\rm for }~k\geq1.$ Then, define $$n_{k+1}:=n'_{i_{k+1}},\quad m_{k+1}:=m'_{i_{k+1}}.$$ Note that $\limsup_{i\rightarrow\infty}(m'_i-n'_i)=\infty$, the sequence $\{i_k:k\geq1\}$ is well defined. By this way, we obtain the subsequences $\{n_k:k\geq1\}$ and $\{m_k:k\geq1\}$ of $\left\{n'_i:i\geq 1\right\}$ and $\left\{m'_i:i\geq1\right\}$, respectively, such that the sequence $\{m_k-n_k: k\geq1\}$ is non-decreasing. As the similar discussion in \cite{YLiao2016}, one has $$\hat{\nu}_\beta(x)=\liminf_{k\rightarrow\infty}\dfrac{m_k-n_k}{n_{k+1}}\leq1.$$ This contradicts our assumption $\hat{\nu}_\beta(x)=\infty$. Thus, we have proved $$\{x\in[0,1]:\hat{\nu}_\beta(x)=\infty\}\subseteq\cup_{n=1}^{\infty}\cup_{\omega\in\Sigma^n_\beta}\{x\in[0,1]:d_\beta(x)=(\omega,0^\infty)\}.$$ Therefore, $$\{x\in[0,1]:\hat{\nu}_\beta(x)=\infty\}=\cup_{n=1}^{\infty}\cup_{\omega\in\Sigma^n_\beta}\{x\in[0,1]:d_\beta(x)=(\omega,0^\infty)\},$$ which implies that the set $\{x\in[0,1]:\hat{\nu}_\beta(x)=\infty\}$ is countable.
\end{proof}   

\begin{Lemma}\label{empty}
	The set $\{x\in[0,1]:1<\hat{\nu}_\beta(x)<\infty\}$ is empty.
\end{Lemma} 

\begin{proof}
   This follows from the proof of Item $(2)$ of Lemma \ref{infinite}.
\end{proof}

\begin{Proposition}\label{uncountable}
   The sets $$\{x\in[0,1]:\nu_\beta(x)=\infty\},\quad \{x\in[0,1]:\hat{\nu}_\beta(x)=1\}$$ are uncountable.  	
\end{Proposition}       
 
\begin{proof}
   For any real number $a>1$, we give a correspondence: $$\Psi(a)\mapsto 10^{\left\lfloor 2^a\right\rfloor}10^{\left\lfloor 2^{a^{2^2}}\right\rfloor}10^{\left\lfloor2^{a^{3^2}}\right\rfloor}1\cdots10^{\left\lfloor2^{a^{k^2}}\right\rfloor}1\cdots.$$ The infinite string is a $\beta$-expansion of some $x\in[0,1)$. Denote this $x$ by $x_a$. Then, we can obtain a correspondence: $$\Phi(a)\mapsto x_a$$ from $\{a: a>1\}$ to $\{x_a\}_{a>1}$. One can check $\nu_\beta(x_a)=\infty$ and $\hat{\nu}_\beta(x_a)=1$. Therefore, $$\{x_a \}_{a>1}\subseteq\{x\in[0,1]:\nu_\beta(x)=\infty\},\quad\{x_a\}_{a>1}\subseteq\{x\in[0,1]:\hat{\nu}_\beta(x)=1\}.$$ 
   
   For different $a_1>1$ and $a_2>1$, there is a $k_0\in\mathbb{N}^+$ such that $$\left|2^{a_1^{k^2}}-2^{a_2^{k^2}}\right|>1, \text{ for any }k\geq k_0.$$ Thus, $\Psi(a_1)\neq\Psi(a_2)$. Then, $\Phi(a_1)\neq\Phi(a_2)$. Hence, the cardinality of $\{a: a>1\}$ is less than or equal to that of $\{x_a\}_{a>1}$. Similarly, the cardinality of $\{x_a: a>1\}$ is less than or equal to that of $\{x\in[0,1]:\nu_\beta(x)=\infty\}$ ($\{x\in[0,1]:\hat{\nu}_\beta(x)=1\}$). Since $\{a: a>1\}$ is uncountable, $\{x\in[0,1]:\nu_\beta(x)=\infty\}$ and $\{x\in[0,1]:\hat{\nu}_\beta(x)=1\}$ are uncountable.      
\end{proof}   
  
\begin{proof}[Proof of Theorem \ref{Special-case}]
   $(1)$ If $\underline{v}_1=\overline{v}_1=\underline{v}_2=\overline{v}_2=0$, then $$\lim_{n\rightarrow\infty}\dfrac{-\log_\beta\psi_1(n)}{n}=0,\qquad\lim_{n\rightarrow\infty}\dfrac{-\log_\beta\psi_2(n)}{n}=0.$$ Thus, for any positive integer $m$ large enough, there is an integer $n_0>0$ such that $$\beta^{-n/m}\leq\psi_1(n),\quad \beta^{-n/m}\leq\psi_2(n),\quad \text{for any $n\geq n_0$}.$$ Therefore, $$\{x\in[0,1]:\nu_\beta(x)\geq1/m\}\subseteq\mathcal{L}(\psi_1),\quad\{x\in[0,1]:\hat{\nu}_\beta(x)=1/m\}\subseteq\mathcal{U}(\psi_2).$$ By the fact $\{x\in[0,1]:\hat{\nu}_\beta(x)=1/m\}\subseteq\{x\in[0,1]:\nu_\beta(x)\geq1/m\}$, one has $$\{x\in[0,1]:\hat{\nu}_\beta(x)=1/m\}\subseteq\mathcal{L}(\psi_1)\cap\mathcal{U}(\psi_2).$$ According to \cite[Theorem 1.5]{YLiao2016}, we have $$1=\sup_{m\in\mathbb{N}^+}{\rm dim}_H\left(\{x\in[0,1]:\hat{\nu}_\beta(x)=1/m\}\right)\leq{\rm dim}_H\left(\mathcal{L}(\psi_1)\cap\mathcal{U}(\psi_2)\right).$$ Thus, ${\rm dim}_H\left(\mathcal{L}(\psi_1)\cap\mathcal{U}(\psi_2)\right)=1.$
	
   $(2)$ If $\underline{v}_2=\infty$, then $$\lim_{n\rightarrow\infty}\dfrac{-\log_\beta\psi_2(n)}{n}=\infty.$$ For any $L>0$ large enough, there is an integer $n_0$ such that $$\psi_2(n)\leq\beta^{-nL},~\text{ for any $n\geq n_0$}.$$ Therefore, for any $x\in\mathcal{U}(\psi_2)$, one has $\hat{\nu}_\beta(x)\geq L$. By the arbitrariness of $L$, $\hat{\nu}_\beta(x)=\infty$. Thus, $\mathcal{U}(\psi_2)\subseteq\{x\in[0,1]:\hat{\nu}_\beta(x)=\infty\}$. By  Lemma \ref{infinite} (2), $$\mathcal{U}(\psi_2)=\cup_{n=1}^{\infty}\cup_{\omega\in\Sigma^n_\beta}\{x\in[0,1]:d_\beta(x)=(\omega,0^\infty)\}=\{x\in[0,1]:\hat{\nu}_\beta(x)=\infty\}.$$ According to Lemma \ref{infinite} (1), $$\mathcal{L}(\psi_1)\cap\mathcal{U}(\psi_2)=\{x\in[0,1]:\hat{\nu}_\beta(x)=\infty\}.$$ Since $\{x\in[0,1]:\hat{\nu}_\beta(x)=\infty\}$ is countable, $\mathcal{L}(\psi_1)\cap\mathcal{U}(\psi_2)$ is countable.
	
   $(3)$ If $\underline{v}_1=\infty$, by {\bf Theorem SW}, $${\rm dim}_H(\mathcal{L}(\psi_1)\cap\mathcal{U}(\psi_2))\leq{\rm dim}_H(\mathcal{L}(\psi_1))=0.$$	 
\end{proof}  
 
   Now, we prove Theorem \ref{Nonspecial}. We divide the proof into three propositions as Propositions \ref{Thm3.1}, \ref{Thm3.2}, \ref{Thm3.3}.  
\begin{Proposition}\label{Thm3.1}
   For any $0\leq\underline{v}_2\leq\overline{v}_2<\infty$, one has $${\rm dim}_H\left( \mathcal{L}(\psi_1)\cap\mathcal{U}(\psi_2)\right)\leq\dfrac{1}{1+\overline{v}_2}.$$
\end{Proposition}

\begin{proof}
   By the definition of $\overline{v}_2$, take a subsequence $\{n_k:k\geq1\}$ such that $$\lim_{k\rightarrow\infty}\dfrac{-\log_\beta\psi_2(n_k)}{n_k}=\overline{v}_2.$$ Then, for any $\varepsilon>0$, there is an integer $k_0$ such that $$\beta^{-n_k(\overline{v}_2+\varepsilon)}\leq\psi_2(n_k)\leq\beta^{-n_k(\overline{v}_2-\varepsilon)},~\text{for any $k\geq k_0$}.$$ For every $x\in\mathcal{U}(\psi_2)$, by the same argument as Proposition \ref{Pro}, we have $$\mathcal{U}(\psi_2)\subseteq\left\{x\in[0,1]:T_\beta^n x <\beta^{-n(\overline{v}_2-\varepsilon)},\text{ for infinitely many $n\in\mathbb{N}$}\right\}.$$ By \cite[Theorem 1.1]{SW2013}, ${\rm dim}_H\left(\mathcal{U}(\psi_2)\right)\leq\dfrac{1}{1+\overline{v}_2}$. Then, $${\rm dim}_H\left( \mathcal{L}(\psi_1)\cap\mathcal{U}(\psi_2)\right)\leq{\rm dim}_H\left(\mathcal{U}(\psi_2)\right)\leq\dfrac{1}{1+\overline{v}_2}.$$     
\end{proof}   
   
\begin{Proposition}\label{Thm3.2}
   If $\underline{v}_2>1$, then the set $\mathcal{L}(\psi_1)\cap\mathcal{U}(\psi_2)$ is countable. If $\underline{v}_1/(2+\underline{v}_1)<\underline{v}_2\leq1$, then $${\rm dim}_H \left(\mathcal{L}(\psi_1)\cap\mathcal{U}(\psi_2)\right)\leq\left(\dfrac{1-\underline{v}_2}{1+\underline{v}_2}\right)^2.$$ If $\underline{v}_2\leq \underline{v}_1/(2+\underline{v}_1)$, then $${\rm dim}_H\left( \mathcal{L}(\psi_1)\cap\mathcal{U}(\psi_2)\right)\leq \dfrac{\underline{v}_1-\underline{v}_2-\underline{v}_1\cdot\underline{v}_2}{(1+\underline{v}_1)(\underline{v}_1-\underline{v}_2)}.$$     
\end{Proposition}

\begin{proof}
   By Lemma \ref{SET} $(2)$, $$\mathcal{L}(\psi_1)\cap\mathcal{U}(\psi_2)\subseteq\{x\in[0,1]:\nu_\beta(x)\geq\underline{v}_1\}\cap\{x\in[0,1]:\hat{\nu}_\beta(x)\geq\underline{v}_2\}.$$ The argument on the upper bound of the Hausdorff dimension of the set $\mathcal{L}(\psi_1)\cap\mathcal{U}(\psi_2)$ can be obtained by a natural covering of the set $$\{x\in[0,1]:\nu_\beta(x)\geq \underline{v}_1\}\cap\{x\in[0,1]:\hat{\nu}_\beta(x)\geq\underline{v}_2\}.$$ According to Theorem \ref{Special-case} $(2)$ and $(3)$, we only need to consider the case $\nu_\beta(x)\in[\underline{v}_1,\infty)$ and $\hat{\nu}_\beta(x)\in[\underline{v}_2,\infty)$. For any $x\in\mathcal{L}(\psi_1)\cap\mathcal{U}(\psi_2)$, there is a number $v_\beta\in[\underline{v}_1,\infty)$ such that $$x\in\mathbb{B}:=\{x\in[0,1]:\nu_\beta(x)=v_\beta\}\cap\{x\in[0,1]:\hat{\nu}_\beta(x)\geq\underline{v}_2\}.$$ Denote its $\beta$-expansion by $$x=\dfrac{a_1}{\beta}+\dfrac{a_2}{\beta^2}+\cdots+\dfrac{a_n}{\beta^n}+\cdots,$$ where $a_i\in\{0,\cdots,\lceil\beta\rfloor\}$, for all $i\geq1$. Since $\nu_\beta(x)<\infty$, $T_\beta^nx>0$ for all $n\geq0$. By the same way as Lemma \ref{infinite}, we take the maximal subsequences $\{n_k:k\geq1\}$ and $\{m_k:k\geq1\}$ of $\left\{n'_i:i\geq 1\right\}$ and $\left\{m'_i:i\geq1\right\}$, respectively. Notice $\beta^{n_k-m_k}<T^{n_k}_\beta x<\beta^{n_k-m_k+1}$. We have the following claim.
   
\begin{claim}
   \quad $v_\beta=\limsup\limits_{k\rightarrow\infty}\dfrac{m_k-n_k}{n_k},\quad \underline{v}_2\leq\liminf\limits_{k\rightarrow\infty}\dfrac{m_k-n_k}{n_{k+1}}.$
\end{claim}
   
\begin{proof}[Proof of Claim]
   Without loss of generality, we assume $$\limsup\limits_{k\rightarrow\infty}\dfrac{m_k-n_k}{n_k}=c_1,\quad\liminf\limits_{k\rightarrow\infty}\dfrac{m_k-n_k}{n_{k+1}}=c_2.$$ First, we show $v_\beta=c_1$. For any $\varepsilon>0$, there is an integer $k_0>0$ such that $$m_k-n_k\leq n_k(c_1+\varepsilon),\quad\text{for any $k\geq k_0$}.$$ Since $T^{n_k}_\beta x>\beta^{n_k-m_k}$, we have $T^{n_k}_\beta x>\beta^{n_k-m_k}\geq\beta^{-n_k(c_1+\varepsilon)}$. In general, for any $n\geq n_{k_0}$, there is an integer $k\geq k_0$ such that $n_k\leq n<n_{k+1}$. By the choice of $\{n_k\}$, we have $$T^n_\beta x>T^{n_k}_\beta x>\beta^{-n_k(c_1+\varepsilon)}>\beta^{-n(c_1+\varepsilon)}.$$ It means $v_\beta=v_\beta(x)<c_1+\varepsilon$. On the other hand, by the definition of $c_1$, taking subsequence $\{n_{k_i}\}$ and $\{m_{k_i}\}$ such that $$\dfrac{m_{k_i}-n_{k_i}}{n_{k_i}}\geq c_1-\varepsilon,$$ one has $T^{n_{k_i}}_\beta x<\beta^{n_{k_i}-m_{k_i}+1}\leq\beta^{-n_{k_i}(c_1-\varepsilon)+1}$. Thus, $v_\beta=\nu_\beta(x)\geq c_1-\varepsilon$. By the arbitrariness of $\varepsilon$, $v_\beta=c_1$.
   
   Next, we will prove $\underline{v}_2\leq c_2$. By the definition of $\underline{v}_2$, for any $\varepsilon>0$, there is an integer $n_0=n_0(\varepsilon)>0$ such that  $$\psi_2(n)\leq\beta^{-n(\underline{v}_2-\varepsilon)},\quad\text{for any $n\geq n_0$}.$$ By the definition of $c_2$, one can take a subsequence $\{k_i:i\geq 1\}$ such that $$\lim_{i\rightarrow\infty}\dfrac{m_{k_i}-n_{k_i}}{n_{k_{i+1}}}=c_2.$$ For the above $\varepsilon>0$, there is an integer $i_0=i_0(\varepsilon)>0$ such that $$ m_{k_i}-n_{k_i}\leq n_{k_{i+1}}(c_2+\varepsilon),\quad\text{for any $i\geq i_0$}.$$ By contrary, suppose $c_2<\underline{v}_2$. Then, for any $\varepsilon\in(0,(\underline{v}_2-c_2)/4)$ and any integer $J\geq K:=\max\{n_0(\varepsilon),~n_{i_0(\varepsilon)}\}$, there is an integer $n_{k_{i+1}}>J$ such that for any integer $n\in[1,n_{k_{i+1}}]$, one has $$T^n_\beta x>T^{n_{k_i}}_\beta x>\beta^{n_{k_i}-m_{k_i}}\geq\beta^{-n_{k_{i+1}}(c_2+\varepsilon)}>\beta^{-n_{k_{i+1}}(\underline{v}_2-\varepsilon)}\geq\psi_2(n_{k_{i+1}}).$$ This contracts the fact $x\in\mathcal{U}(\psi_2)$. Therefore, $$\underline{v}_2\leq c_2=\liminf\limits_{k\rightarrow\infty}\dfrac{m_k-n_k}{n_{k+1}}.$$ 
\end{proof}

   Now, we consider
\begin{equation}\label{E2}
   v_\beta=\limsup_{n\rightarrow\infty}\dfrac{m_k-n_k}{n_k}=\limsup_{n\rightarrow\infty}\dfrac{m_k}{n_k}-1,
\end{equation} 
   
\begin{equation}\label{E3}
   \underline{v}_2\leq\liminf_{n\rightarrow\infty}\dfrac{m_k-n_k}{n_{k+1}}\leq\liminf_{n\rightarrow\infty}\dfrac{m_k-n_k}{m_k}=1-\limsup_{n\rightarrow\infty}\dfrac{n_k}{m_k}.
\end{equation}
   Since $\left(\limsup\dfrac{n_k}{m_k}\right)\cdot\left(\limsup\dfrac{m_k}{n_k}\right)\geq 1$, one has 
\begin{equation}\label{in1}
   v_\beta\geq\dfrac{\underline{v}_2}{1-\underline{v}_2},\quad \underline{v}_2\leq\dfrac{v_\beta}{1+v_\beta}.
\end{equation}

   If $\underline{v}_2>1$, then $\underline{v}_2>v_\beta/(1+v_\beta)$, for any $v_\beta\geq v_1$. This contradicts (\ref{in1}). Thus, $\mathbb{B}$ is empty. By Lemma \ref{infinite}, $$\mathcal{L}(\psi_1)\cap\mathcal{U}(\psi_2)=\cup_{n=1}^{\infty}\cup_{\omega\in\Sigma^n_\beta}\{x\in[0,1]:d_\beta(x)=(\omega,0^\infty)\}.$$ Hence, $\mathcal{L}(\psi_1)\cap\mathcal{U}(\psi_2)$ is countable.
   
   If $\underline{v}_2\leq1$, by the inequality (\ref{in1}), for any $v_\beta<\underline{v}_2/(1-\underline{v}_2)$, the set $\mathbb{B}$ is empty. Therefore, we consider the case $v_\beta\geq\underline{v}_2/(1-\underline{v}_2)$. Take a subsequence $\{k_i:i\geq 1\}$ such that the supremum of (\ref{E2}) is obtained. For abbreviation, we continue to write $\{n_k:k\geq1\}$ and $\{m_k:k\geq1\}$ for the subsequence $\{n_{k_i}:i\geq 1\}$ and $\{m_{k_i}:i\geq1\}$, respectively. Given $0<\varepsilon<\underline{v}_2/2$, for $k$ large enough, one has
\begin{equation}\label{E4}
   	(v_\beta-\varepsilon)n_k\leq m_k-n_k\leq(v_\beta+\varepsilon)n_k,
\end{equation}   
\begin{equation}\label{E5}
   	m_k-n_k\geq(\underline{v}_2-\varepsilon)n_{k+1}.
\end{equation}
   By inequality (\ref{E4}), one has $$(1+v_\beta-\varepsilon)m_{k-1}\leq(1+v_\beta-\varepsilon)n_k\leq m_k.$$ Therefore, the sequence $\{m_k:k\geq1\}$ increases at least exponentially. Since $n_k\geq m_{k-1}$ for every $k\geq2$, the sequence $\{n_k:k\geq1\}$ also increases at least exponentially. Thus, there is a positive constant $C$ such that $k\leq C\log_\beta n_k$. Combining (\ref{E4}) and (\ref{E5}), one obtains $$(\underline{v}_2-\varepsilon)n_{k+1}\leq(v_\beta+\varepsilon)n_k.$$ Thus, for $k$ large enough, there is an integer $n_0$ and a postive real number $\varepsilon_1$ small enough such that the sum of all lengths of the blocks of $0$ in the prefix of length $n_k$ of the infinite sequence $a_1a_2\cdots$ is at least equal to 
\begin{eqnarray*}
   	\sum^k_{i=1}(\underline{v}_2-\varepsilon)n_i-n_0&=&n_k(\underline{v}_2-\varepsilon)\left\lgroup \dfrac{1-\left(\dfrac{\underline{v}_2-\varepsilon}{v_\beta+\varepsilon}\right)^k}{1-\left(\dfrac{\underline{v}_2-\varepsilon}{v_\beta+\varepsilon}\right)}\right\rgroup-n_0\\ &\geq& n_k\left(\dfrac{v_\beta \cdot \underline{v}_2}{v_\beta-\underline{v}_2}-\varepsilon_1\right).
\end{eqnarray*} 
   Among the digits $a_1\cdots a_{m_k}$, there are $k$ blocks of digits which are \textquoteleft free\textquoteright. Denote their lengths by $l_1,\cdots,l_k$. Let $\varepsilon_2=\dfrac{(v_\beta-\underline{v}_2-v_\beta\cdot \underline{v}_2)\varepsilon_1}{v_\beta-\underline{v}_2}$, one has $$\sum_{i=1}^kl_i\leq n_k-n_k\left(\dfrac{v_\beta\cdot \underline{v}_2}{v_\beta-\underline{v}_2}-\varepsilon_1\right)= n_k(1+\varepsilon_2)\dfrac{v_\beta-\underline{v}_2-v_\beta\cdot\underline{v}_2}{v_\beta-\underline{v}_2}.$$ By Theorem \ref{cardinality}, there are at most $\beta\cdot\beta^{l_i}/(\beta-1)$ ways to choose the block with length $l_i$. Thus, one has in total at most $$\left(\dfrac{\beta}{\beta-1}\right)^k\cdot\beta^{\sum_{i=1}^kl_i}\leq \left(\dfrac{\beta}{\beta-1}\right)^k\cdot\beta^{n_k(1+\varepsilon_2)(v_\beta-\underline{v}_2-v_\beta\cdot \underline{v}_2)/(v_\beta-\underline{v}_2)}$$ possible choices of the digits $a_1\cdots a_{m_k}$. On the other hand, there are at most $k(k\leq C\log_\beta n_k)$ blocks of $0$ in the prefix of length $n_k$ of the infinite sequence $a_1a_2\cdots$. Since there are at most $n_k$ possible choices for their first index, one has in total at most $(n_k)^{C\log_\beta n_k}$ possible choices. Consequently, the set of those $x\in\mathbb{B}$ is covered by $$\left(\dfrac{\beta n_k}{\beta-1}\right)^{C\log_\beta n_k}\cdot\beta^{n_k(1+\varepsilon_2)(v_\beta-\underline{v}_2-v_\beta\cdot\underline{v}_2)/(v_\beta-\underline{v}_2)}$$ basic intervals of length at most $\beta^{-m_k}$. Moreover, by (\ref{E4}) and by letting $\varepsilon_3=\varepsilon/(1+v_\beta)$, we have $\beta^{-m_k}\leq \beta^{-(1+v_\beta)(1-\varepsilon_3)n_k}$. Set $\varepsilon'=\max\{\varepsilon_2,~\varepsilon_3\}$. The set of those $x$ is covered by $$\left(\dfrac{\beta n_k}{\beta-1}\right)^{C\log_\beta n_k}\cdot\beta^{n_k(1+\varepsilon')(v_\beta-\underline{v}_2-v_\beta\cdot \underline{v}_2)/(v_\beta-\underline{v}_2)}$$ basic intervals of length at most $\beta^{-(1+v_\beta)(1-\varepsilon')n_k}$. We consider the series $$\sum_{N\geq1}(N)^{C\log_\beta N}\beta^{N(1+\varepsilon')(v_\beta-\underline{v}_2-v_\beta\cdot \underline{v}_2)/(v_\beta-\underline{v}_2)}\beta^{-(1+v_\beta)(1-\varepsilon')Ns}.$$ The critical exponent $s_0$ such that the series converges if $s>s_0$ and diverges if $s<s_0$ is given by $$s_0=\dfrac{1+\varepsilon'}{1-\varepsilon'}\cdot\dfrac{v_\beta-\underline{v}_2-v_\beta\cdot\underline{v}_2}{(1+v_\beta)(v_\beta-\underline{v}_2)}.$$ By a standard covering argument and the arbitrariness of $\varepsilon'$, the Hausdorff dimension of $\mathbb{B}':=\{x\in[0,1]: \nu_\beta(x)=\nu_\beta\}\cap\mathcal{U}(\psi_2)$ is at most equal to 
\begin{equation}
   {\rm dim}_H \left(\mathbb{B}'\right)\leq\dfrac{v_\beta-\underline{v}_2-v_\beta\cdot\underline{v}_2}{(1+v_\beta)(v_\beta-\underline{v}_2)}.
\end{equation}
   
   For $v_\beta\geq\underline{v}_2/(1-\underline{v}_2)$, fix $L$ large enough. We consider the set $$\mathbb{D}:=\{x\in[0,1]:v_\beta\leq \nu_\beta(x)<v_\beta+1/L\}\cap\mathcal{U}(\psi_2).$$ Repeat the above discussion, if $\underline{v}_2<1$, then $${\rm dim}_H \left(\mathbb{D}\right)\leq\dfrac{v_\beta-\underline{v}_2-v_\beta\cdot v_2}{(1+v_\beta)(v_\beta-\underline{v}_2)}+\dfrac{\underline{v}_2^2/L}{1-\underline{v}_2}.$$ 
   
   If $\underline{v}_2=1$, then $v_\beta=+\infty$. By {\bf Theorem SW}, ${\rm dim}_H \left(\mathbb{D}\right)=0$. If $\underline{v}_2<1$, since $\mathcal{L}(\psi_1)\cap\mathcal{U}(\psi_2)$ is a subset of  $$\cup_{N=0}^{+\infty}\cup^L_{i=1}\left\{x\in[0,1]:v_1+N+(i-1)/L\leq\nu_\beta(x)<v_1+N+i/L\right\}\cap\mathcal{U}(\psi_2),$$ let $L$ tend to $+\infty$, one has $${\rm dim}_H\left( \mathcal{L}(\psi_1)\cap\mathcal{U}(\psi_2)\right)\leq\sup_{v_\beta\geq \underline{v}_2/(1-\underline{v}_2)}\left\{\dfrac{v_\beta-\underline{v}_2-v_\beta\cdot \underline{v}_2}{(1+v_\beta)(v_\beta-\underline{v}_2)}\right\}.$$ Regard the right side as a function of $v_\beta$, if $\underline{v}_1/(2+\underline{v}_1)<\underline{v}_2$, then the maximum is attained for $v_\beta=2\underline{v}_2/(1-\underline{v}_2)$. Therefore, $${\rm dim}_H \left(\mathcal{L}(\psi_1)\cap\mathcal{U}(\psi_2)\right)\leq\left(\dfrac{1-\underline{v}_2}{1+\underline{v}_2}\right)^2.$$ If $\underline{v}_2\leq\underline{v}_1/(2+\underline{v}_1)$, then the maximum is attained for $v_\beta=\underline{v}_1$. Thus, $${\rm dim}_H\left( \mathcal{L}(\psi_1)\cap\mathcal{U}(\psi_2)\right)\leq \dfrac{\underline{v}_1-\underline{v}_2-\underline{v}_1\cdot\underline{v}_2}{(1+\underline{v}_1)(\underline{v}_1-\underline{v}_2)}.$$           
\end{proof}

\begin{Proposition}\label{Thm3.3}
   If $\overline{v}_1/(2+\overline{v}_1)<\overline{v}_2\leq1$, then $${\rm dim}_H \left(\mathcal{L}(\psi_1)\cap\mathcal{U}(\psi_2)\right)\geq\left(\dfrac{1-\overline{v}_2}{1+\overline{v}_2}\right)^2.$$ If $\overline{v}_2\leq\overline{v}_1/(2+\overline{v}_1)$, then $${\rm dim}_H\left( \mathcal{L}(\psi_1)\cap\mathcal{U}(\psi_2)\right)\geq\dfrac{\overline{v}_1-\overline{v}_2-\overline{v}_1\cdot\overline{v}_2}{(1+\overline{v}_1)(\overline{v}_1-\overline{v}_2)}.$$ 
\end{Proposition}

\begin{proof}
   By Lemma \ref{SET} $(1)$, $$\{x\in[0,1]:\nu_\beta(x)>\overline{v}_1\}\cap\{x\in[0,1]:\hat{\nu}_\beta(x)>\overline{v}_2\}\subseteq\mathcal{L}(\psi_1)\cap\mathcal{U}(\psi_2).$$ If $\overline{v}_2=1$, then ${\rm dim}_H \left(\mathcal{L}(\psi_1)\cap\mathcal{U}(\psi_2)\right)\geq0$ always holds. If $\overline{v}_2<1$, then we fix $\delta>0$ with $\overline{v}_2+\delta<1$, we consider the lower bound of the Hausdorff dimension of the set $$\mathbb{F}:=\{x\in[0,1]:\nu_\beta(x)=v_\beta+\delta\}\cap\{x\in[0,1]:\hat{\nu}_\beta(x)\geq\overline{v}_2+\delta\},$$ where $v_\beta\geq\overline{v}_1$ is a real number. By {\bf Theorem BL}, if $\dfrac{\overline{v}_1+\delta}{\overline{v}_2+\delta}<\dfrac{1}{1-(\overline{v}_2+\delta)}$, then $\mathbb{F}$ is empty. Therefore, we consider the case $\dfrac{\overline{v}_1+\delta}{\overline{v}_2+\delta}\geq\dfrac{1}{1-(\overline{v}_2+\delta)}$. If $\overline{v}_2>0$, then there is $\delta_0>0$ such that for any $\delta\in(0,\delta_0]$, one has   $$\dfrac{v_\beta+\delta}{\overline{v}_2+\delta}\geq\dfrac{\overline{v}_1+\delta}{\overline{v}_2+\delta}\geq\dfrac{1}{1-(\overline{v}_2+\delta)}>1.$$ For any $\delta\in(0,\delta_0]$, we will construct a Cantor subset $E_\delta$ of $\mathbb{F}$. Let $$n'_k=\left\lfloor\left(\dfrac{v_\beta+\delta}{\overline{v}_2+\delta}\right)^k\right\rfloor,\quad m'_k=\lfloor(1+v_\beta+\delta)n'_k\rfloor,\quad k=1,2,\cdots$$ If $\overline{v}_2=0$, let $$n'_k=k^k,\quad m'_k=\lfloor(1+ v_\beta+\delta)n'_k\rfloor,\quad k=1,2,\cdots$$ Making an adjustment, we can choose two subsequences $\{n_k\}$ and $\{m_k\}$ with $n_k<m_k<n_{k+1}$ for every $k\geq1$ such that $\{m_k-n_k\}$ is a non-decreasing sequence and 
\begin{equation}\label{lowle}
   \lim_{k\rightarrow\infty}\dfrac{m_k-n_k}{n_k}=v_\beta+\delta,\qquad \lim_{k\rightarrow\infty}\dfrac{m_k-n_k}{n_{k+1}}=\overline{v}_2+\delta.
\end{equation}
   
   Consider the set of real numbers $x\in[0,1)$ whose $\beta$-expansion $$x=\dfrac{a_1}{\beta}+\dfrac{a_2}{\beta^2}+\cdots+\dfrac{a_n}{\beta^n}+\cdots,$$ satisfies that for all $k\geq1$, $$a_{n_k}=1,~a_{n_k+1}=\cdots=a_{m_k-1}=0,~a_{m_k}=1,$$ $$ a_{m_k+(m_k-n_k)}=a_{m_k+2(m_k-n_k)}=\cdots=a_{m_k+t_k(m_k-n_k)}=1,$$ where $t_k$ is the largest integer such that $m_k+t_k(m_k-n_k)<n_{k+1}$. Then, $$t_k\leq\dfrac{n_{k+1}-m_k}{m_k-n_k}\leq\dfrac{2}{\overline{v}_2+\delta},$$ for $k$ large enough. Therefore, the sequence $\{t_k:k\geq1\}$ is bounded. Fix $N$, let $\beta_N$ be the real number defined by the infinite $\beta$-expansion of $1$ as equality (\ref{ED1}). We replace the digit $1$ for $a_{n_k},~a_{m_k}$ and $a_{m_k+i(m_k-n_k)}$ for any $1\leq i\leq t_k$ by the block $0^N10^N$. Fill other places by blocks belonging to $\Sigma_{\beta_N}$. Thus, we have constructed the Cantor type subset $E_\delta$. Since $\{t_k\}$ is bound, one has $$\lim_{k\rightarrow\infty}\dfrac{m_k-n_k-1+2N}{n_k+(4k-2)N+\sum_{i=1}^{k-1}2Nt_i}=\lim_{k\rightarrow\infty}\dfrac{m_k-n_k}{n_k}=v_\beta+\delta,$$ $$\lim_{k\rightarrow\infty}\dfrac{m_k-n_k-1+2N}{n_{k+1}+(4k+2)N+\sum_{i=1}^k2Nt_i}=\lim_{k\rightarrow\infty}\dfrac{m_k-n_k}{n_{k+1}}=\overline{v}_2+\delta.$$ According to the construction, the sequence $d_\beta(x)$ is in $\Sigma_{\beta_N}.$
\begin{claim}
   	\qquad $E_\delta\subseteq \mathcal{L}(\psi_1)\cap\mathcal{U}(\psi_2)$.
\end{claim}
   
\begin{proof}[Proof of Claim]
   Given $\varepsilon>0$, by (\ref{lowle}), there exists an integer $k_0$ such that $$m_k-n_k\leq(v_\beta+\delta+\varepsilon)n_k,~ m_k-n_k\leq(\overline{v}_2+\delta+\varepsilon)n_{k+1},~\text{for any $k\geq k_0$}.$$ By the definitions of $\overline{v}_1$ and $\overline{v}_2$, there is an integer $n_0$ such that $$\beta^{-n(\overline{v}_1+\delta+\varepsilon)}\leq\psi_1(n),\quad \beta^{-n(\overline{v}_2+\delta+\varepsilon)}\leq\psi_2(n),~\text{for any $n\geq n_0$}.$$
   
   Let $N_0=\max\{n_{k_0},~n_0\}$, for any $x\in E_\delta$ and any $n_k\geq N_0$, one has $$T^{n_k}_\beta x<\beta^{n_k-m_k+1}\leq\beta^{-n_k(v_\beta+\delta+\varepsilon-1/n_k)}\leq\beta^{-n_k(\overline{v}_1+\delta+\varepsilon-1/n_k)}\leq\psi_1(n_k).$$ It means $x\in\mathcal{L}(\psi_1)$. On the other hand, for $N\geq N_0$, there is an integer $i$ such that $n_{k+i}\leq N<n_{k+i+1}$. Therefore, $$T^{n_{k+i}}_\beta x<\beta^{n_{k+i}-m_{k+i}+1}\leq\beta^{-N(\overline{v}_2+\delta+\varepsilon-1/n_{k+i+1})}\leq\psi_2(N).$$ It means $x\in\mathcal{U}(\psi_2)$. Then, $x\in\mathcal{L}(\psi_1)\cap\mathcal{U}(\psi_2)$. Therefore, $$E_\delta\subseteq \mathcal{L}(\psi_1)\cap\mathcal{U}(\psi_2).$$  	
\end{proof}

   We distribute the mass uniformly when meet a block in $\Sigma_{\beta_N}$ and keep the mass when go through the positions where the digits are determined by construction of $E_\delta$. The Bernoulli measure $\mu$ on $E_\delta$ is defined as follows.
   
   If $n<n_1$, define $\mu(I_n)=1/\sharp\Sigma^n_{\beta_N}$. If $n_1\leq n\leq m_1+4N$, define $\mu(I_n)=1/\sharp\Sigma^{n_1-1}_{\beta_N}$. If there is an integer $t$ with $0\leq t\leq t_1-1$ such that $$m_1+4N+(t+1)(m_1-n_1)+2Nt<n\leq m_1+4N+(t+1)(m_1-n_1)+2N(t+1),$$ define $$\mu(I_n)=\dfrac{1}{\sharp\Sigma^{n_1-1}_{\beta_N}}\cdot\dfrac{1}{\left(\sharp\Sigma^{m_1-n_1-1}_{\beta_N}\right)^{t+1}}.$$ If there is an integer $t$ with $0\leq t\leq t_1$ such that $$m_1+4N+t(m_1-n_1)+2Nt<n\leq c,$$ where $c:=\min\{n_2+4N+2Nt_1, m_1+4N+(t+1)(m_1-n_1)+2Nt\}$, define $$\mu(I_n)=\dfrac{1}{\sharp\Sigma^{n_1-1}_{\beta_N}}\cdot\dfrac{1}{\left(\sharp\Sigma^{m_1-n_1-1}_{\beta_N}\right)^t}\cdot\dfrac{1}{\sharp\Sigma^{n-(m_1+4N+t(m_1-n_1)+2Nt)}_{\beta_N}}.$$ 
   
   For $k\geq2$, let $$l_k:=n_k+4(k-1)N+\sum^{k-1}_{i=1}2Nt_i,\quad h_k:=m_k+4kN+\sum^{k-1}_{i=1}2Nt_i,$$ $$p_k:=m_k-n_k-1,\quad q_k:=h_k+t_k(m_k-n_k)+2Nt_k.$$ If $l_k\leq n\leq h_k$, define $$\mu(I_n)=\dfrac{1}{\sharp\Sigma^{n_1-1}_{\beta_N}}\cdot\dfrac{1}{\prod^{k-1}_{i=1}\left(\sharp \Sigma^{p_i}_{\beta_N}\right)^{t_i}\cdot\left(\sharp\Sigma^{l_{i+1}-q_i-1}_{\beta_N}\right)}=\mu(I_{l_k})=\mu(I_{h_k}).$$ If there is an integer $t$ with $0\leq t\leq t_k-1$ such that $$h_k+(t+1)(m_k-n_k)+2Nt<n\leq h_k+(t+1)(m_k-n_k)+2N(t+1),$$ define $$\mu(I_n)=\mu(I_{h_k})\cdot\dfrac{1}{\left(\sharp\Sigma^{p_k}_{\beta_N}\right)^{t+1}}.$$ If there is an integer $t$ with $0\leq t\leq t_k$ such that $$h_k+t(m_k-n_k)+2Nt<n\leq\min\{l_{k+1}, h_k+(t+1)(m_k-n_k)+2Nt\},$$ define $$\mu(I_n)=\mu(I_{h_k})\cdot\dfrac{1}{\left(\sharp\Sigma^{p_k}_{\beta_N}\right)^t}\cdot\dfrac{1}{\sharp\Sigma^{n-(h_k+t(m_k-n_k)+2Nt)}_{\beta_N}}.$$ 
   
   By the construction and Proposition \ref{full}, $I_{h_k}$ is full. For calculating the local dimension of $\mu$, we discuss different cases as follows.
   
   {\bf Case $A$:} If $n=h_k$, then
\begin{eqnarray*}
   	\liminf_{k\rightarrow\infty}\dfrac{\log_\beta\mu(I_{h_k})}{\log_\beta \lvert I_{h_k}\rvert}&=&\liminf_{k\rightarrow\infty}\dfrac{n_1-1+\sum\limits_{i=1}^{k-1}\left(t_ip_i+l_{i+1}-q_i-1\right)}{h_k}\cdot\log_\beta\beta_N\\&=&\liminf_{k\rightarrow\infty}\dfrac{n_1-1+\sum\limits_{i=1}^{k-1}\left(l_{i+1}-h_i-2Nt_i-1\right)}{h_k}\cdot\log_\beta\beta_N.
\end{eqnarray*} 
   Recall that $\{t_k:k\geq1\}$ is bounded and $\{m_k:k\geq1\}$ grows exponentially fast in terms of $k$, therefore, $$\liminf_{k\rightarrow\infty}\dfrac{\log_\beta\mu(I_{h_k})}{\log_\beta \lvert I_{h_k}\rvert}=\liminf_{k\rightarrow\infty}\dfrac{\sum_{i=1}^{k-1}\left(n_{i+1}-m_i\right)}{m_k}\log_\beta\beta_N.$$ By equalities (\ref{lowle}), one has $$\lim_{k\rightarrow\infty}\dfrac{m_k}{n_k}=1+v_\beta+\delta, \lim_{k\rightarrow\infty}\dfrac{m_{k+1}}{m_k}=\dfrac{v_\beta+\delta}{\overline{v}_2+\delta},\lim_{k\rightarrow\infty}\dfrac{n_{k+1}}{m_k}=\dfrac{v_\beta+\delta}{(\overline{v}_2+\delta)(1+v_\beta+\delta)}.$$ According to Stolz-Ces\`{a}ro Theorem,
\begin{eqnarray*}
   \liminf_{k\rightarrow\infty}\dfrac{\sum\limits_{i=1}^{k-1}\left(n_{i+1}-m_i\right)}{m_k}&=&\liminf_{k\rightarrow\infty}\dfrac{n_{k+1}-m_k}{m_{k+1}-m_k}\\&=&\liminf_{k\rightarrow\infty}\dfrac{\dfrac{n_{k+1}}{m_k}-1}{\dfrac{m_{k+1}}{m_k}-1}=\dfrac{v_\beta-\overline{v}_2-(v_\beta+\delta) (\overline{v}_2+\delta)}{(1+v_\beta+\delta)(v_\beta-\overline{v}_2)}.
\end{eqnarray*} 
   Thus, $$\liminf_{k\rightarrow\infty}\dfrac{\log_\beta\mu(I_{h_k})}{\log_\beta \lvert I_{h_k}\rvert}=\dfrac{v_\beta-\overline{v}_2-(v_\beta+\delta) (\overline{v}_2+\delta)}{(1+v_\beta+\delta)(v_\beta-\overline{v}_2)}\cdot\log_\beta\beta_N.$$
   
   {\bf Case $B$:} For an integer $n$ large enough, if there is $k\geq2$ such that $l_k\leq n\leq h_k$, then $$\liminf_{k\rightarrow\infty}\dfrac{\log_\beta\mu(I_n)}{\log_\beta\lvert I_n\rvert}\geq \liminf_{k\rightarrow\infty}\dfrac{\log_\beta\mu(I_n)}{\log_\beta\lvert I_{h_k}\rvert}=\liminf_{k\rightarrow\infty}\dfrac{\log_\beta\mu(I_{h_k})}{\log_\beta\lvert I_{h_k}\rvert}.$$
   
   {\bf Case $C$:} For $n$, if there is an integer $t$ with $0\leq t\leq t_k-1$ such that $$h_k+(t+1)(m_k-n_k)+2Nt<n\leq h_k+(t+1)(m_k-n_k)+2N(t+1),$$ then one has $$\mu(I_n)\leq\mu(I_{h_k})\cdot\beta^{-(t+1)p_k}_N.$$ Since $I_{h_k}$ is full, by Proposition \ref{fullc}, $\lvert I_n\rvert=\lvert I_{h_k}\rvert\cdot\lvert I_{n-h_k}(\omega')\rvert$, where $\omega'$ is an admissible block in $\Sigma^{n-h_k}_{\beta_N}$. By Lemma \ref{length}, $$\lvert I_n\rvert\geq\lvert I_{h_k}\rvert\cdot\beta^{-(n-h_k+N)}.$$ Hence, $$\dfrac{-\log_\beta\mu(I_n)}{-\log_\beta\lvert I_n\rvert}\geq\dfrac{-\log_\beta\mu(I_{h_k})+(t+1)p_k\log_\beta\beta_N}{-\log_\beta\lvert I_{h_k}\rvert+((t+1)p_k+N(2t+1))}\geq\dfrac{-\log_\beta\mu(I_{h_k})}{-\log_\beta\lvert I_{h_k}\rvert}\cdot\varphi(N),$$ where $\varphi(N)<1$ and $\varphi(N)$ tends to $1$ as $N$ tends to infinity. If there is an integer $t$ with $0\leq t\leq t_k$ such that $$h_k+t(m_k-n_k)+2Nt<n\leq\min\{l_{k+1}, h_k+(t+1)(m_k-n_k)+2Nt\},$$ then letting $l:=n-(h_k+t(m_k-n_k)+2Nt)$, one has $$\mu(I_n)\leq\mu(I_{h_k})\cdot\beta^{-tp_k-l}_N.$$ Since $I_{h_k}$ is full, by Proposition \ref{fullc}, $\lvert I_n\rvert=\lvert I_{h_k}\rvert\cdot\lvert I_{n-h_k}(\omega')\rvert$, where $\omega'$ is an admissible block in $\Sigma^{n-h_k}_{\beta_N}$. By Lemma \ref{length}, $\lvert I_{n-h_k}(\omega')\rvert\geq\beta^{-(n-h_k+N)}.$ Therefore, $$\lvert I_n\rvert\geq\lvert I_{h_k}\rvert\cdot\beta^{-(n-h_k+N)}.$$ Hence, $$\dfrac{-\log_\beta\mu(I_n)}{-\log_\beta\lvert I_n\rvert}\geq\dfrac{-\log_\beta\mu(I_{h_k})+(tp_k+l)\log_\beta\beta_N}{-\log_\beta\lvert I_{h_k}\rvert+(tp_k+l+t+N(2t+1))}\geq\dfrac{-\log_\beta\mu(I_{h_k})}{-\log_\beta\lvert I_{h_k}\rvert}\cdot\varphi(N).$$ 
   
   Therefore, in all cases, $$\liminf_{k\rightarrow\infty}\dfrac{\log_\beta\mu(I_n)}{\log_\beta\lvert I_n\rvert}\geq\dfrac{v_\beta-\overline{v}_2-(v_\beta+\delta) (\overline{v}_2+\delta)}{(1+v_\beta+\delta)(v_\beta-\overline{v}_2)}\cdot\log_\beta\beta_N\cdot\varphi(N).$$ Given a point $x\in E_\delta$, let $r$ be a number with $\lvert I_{n+1}(x)\rvert\leq r<\lvert I_n(x)\rvert$. We consider the ball $B(x,r)$. By Lemma \ref{length}, every $n$-th order basic interval $I_n$ satisfies $\lvert I_n\rvert\geq \beta^{-(n+N)}$. Hence, the ball $B(x,r)$ interests at most $\lfloor 2\beta^N\rfloor+2$ basic intervals of order $n$. On the other hand, $$r\geq\lvert I_{n+1}(x)\rvert\geq\beta^{-(n+1+N)}=\beta^{-(1+N)}\cdot\beta^{-n}\geq\beta^{-(1+N)}\cdot\lvert I_n(x)\rvert.$$ Therefore, $$\liminf_{r\rightarrow0}\dfrac{\log_\beta\mu(B(x,r))}{\log_\beta r}=\liminf_{n\rightarrow\infty}\dfrac{\log_\beta\mu(I_n(x))}{\log_\beta\lvert I_n(x)\rvert}.$$ By the arbitrariness of $\delta\in(0,\delta_0]$, one has $$\liminf_{k\rightarrow\infty}\dfrac{\log_\beta\mu(I_n)}{\log_\beta\lvert I_n\rvert}\geq\dfrac{v_\beta-\overline{v}_2-v_\beta\cdot\overline{v}_2}{(1+v_\beta)(v_\beta-\overline{v}_2)}\cdot\log_\beta\beta_N\cdot\varphi(N).$$ Let $N$ tend to infinity, by Mass Distribution Principle \cite[pp.60]{FK90}, one has 
\begin{equation}\label{LBE}
   {\rm dim}_H\left(\mathcal{L}(\psi_1)\cap\mathcal{U}(\psi_2)\right)\geq \dfrac{v_\beta-\overline{v}_2-v_\beta\overline{v}_2}{(1+v_\beta)(v_\beta-\overline{v}_2)}.
\end{equation}
   Regarding the right side as a function of $v_\beta$ with $v_\beta\geq\underline{v}_2/(1-\underline{v}_2)$, if $\overline{v}_1/(2+\overline{v}_1)<\overline{v}_2$, then the maximum is attained for $v_\beta=2\overline{v}_2/(1-\overline{v}_2)$. Therefore, $${\rm dim}_H \left(\mathcal{L}(\psi_1)\cap\mathcal{U}(\psi_2)\right)\geq\left(\dfrac{1-\overline{v}_2}{1+\overline{v}_2}\right)^2.$$ If $\overline{v}_2\leq\overline{v}_1/(2+\overline{v}_1)$, then the maximum is attained for $v_\beta=\overline{v}_1$. Thus, $${\rm dim}_H\left( \mathcal{L}(\psi_1)\cap\mathcal{U}(\psi_2)\right)\geq \dfrac{\overline{v}_1-\overline{v}_2-\overline{v}_1\cdot\overline{v}_2}{(1+\overline{v}_1)(\overline{v}_1-\overline{v}_2)}.$$                 
\end{proof}

\begin{proof}[Proof of Theorem \ref{Nonspecial}]
  If $\underline{v}_2>1$, by Proposition \ref{Thm3.2}, $\mathcal{L}(\psi_1)\cap\mathcal{U}(\psi_2)$ is countable. If $\underline{v}_1/(2+\underline{v}_1)\leq\underline{v}_2\leq1<\overline{v}_2$, by Proposition \ref{Thm3.1}, $${\rm dim}_H \left(\mathcal{L}(\psi_1)\cap\mathcal{U}(\psi_2)\right)\leq\dfrac{1}{1+\overline{v}_2}.$$ By Proposition \ref{Thm3.2} and the definition of Hausdorff dimension, we have $$0\leq{\rm dim}_H \left(\mathcal{L}(\psi_1)\cap\mathcal{U}(\psi_2)\right)\leq\min\left\{\dfrac{1}{1+\overline{v}_2},~\left(\dfrac{1-\underline{v}_2}{1+\underline{v}_2}\right)^2\right\}.$$
   
  If $\overline{v}_1/(2+\overline{v}_1)<\underline{v}_2\leq\overline{v}_2\leq1$ and $\overline{v}_1/(2+\overline{v}_1)<\overline{v}_2$, by Propositions \ref{Thm3.1} and \ref{Thm3.2}, we also have $${\rm dim}_H \left(\mathcal{L}(\psi_1)\cap\mathcal{U}(\psi_2)\right)\leq\min\left\{\dfrac{1}{1+\overline{v}_2},~\left(\dfrac{1-\underline{v}_2}{1+\underline{v}_2}\right)^2\right\}.$$ Since $\overline{v}_1/(2+\overline{v}_1)<\overline{v}_2$,  according to Propositions \ref{Thm3.3}, $$\left(\dfrac{1-\overline{v}_2}{1+\overline{v}_2}\right)^2\leq{\rm dim}_H \left(\mathcal{L}(\psi_1)\cap\mathcal{U}(\psi_2)\right).$$ Then, $$\left(\dfrac{1-\overline{v}_2}{1+\overline{v}_2}\right)^2\leq{\rm dim}_H \left(\mathcal{L}(\psi_1)\cap\mathcal{U}(\psi_2)\right)\leq\min\left\{\dfrac{1}{1+\overline{v}_2},~ \left(\dfrac{1-\underline{v}_2}{1+\underline{v}_2}\right)^2\right\}.$$ 
 
  If $\underline{v}_2\leq\underline{v}_1/(2+\underline{v}_1)$ and $\overline{v}_1/(2+\overline{v}_1)<\overline{v}_2\leq 1$, by Propositions \ref{Thm3.1}, \ref{Thm3.2} and \ref{Thm3.3}, $$\left(\dfrac{1-\overline{v}_2}{1+\overline{v}_2}\right)^2\leq{\rm dim}_H\left(\mathcal{L}(\psi_1)\cap\mathcal{U}(\psi_2)\right)\leq \dfrac{\underline{v}_1-\underline{v}_2-\underline{v}_1\cdot\underline{v}_2}{(1+\underline{v}_1)(\underline{v}_1-\underline{v}_2)}.$$
   
  If $\underline{v}_1/(2+\underline{v}_1)<\underline{v}_2\leq\overline{v}_2\leq\overline{v}_1/(2+\overline{v}_1)$, combining Proposition \ref{Thm3.1} with Proposition \ref{Thm3.2}, one has $${\rm dim}_H\left(\mathcal{L}(\psi_1)\cap\mathcal{U}(\psi_2)\right)\leq\min\left\{\dfrac{1}{1+\overline{v}_2},~\left(\dfrac{1-\underline{v}_2}{1+\underline{v}_2}\right)^2\right\}.$$ According to Proposition \ref{Thm3.3}, we have $${\rm dim}_H \left(\mathcal{L}(\psi_1)\cap\mathcal{U}(\psi_2)\right)\geq\dfrac{\overline{v}_1-\overline{v}_2-\overline{v}_1\cdot\overline{v}_2}{(1+\overline{v}_1)(\overline{v}_1-\overline{v}_2)}.$$ Therefore, $$\dfrac{\overline{v}_1-\overline{v}_2-\overline{v}_1\cdot\overline{v}_2}{(1+\overline{v}_1)(\overline{v}_1-\overline{v}_2)}\leq{\rm dim}_H \left(\mathcal{L}(\psi_1)\cap\mathcal{U}(\psi_2)\right)\leq\min\left\{\dfrac{1}{1+\overline{v}_2},~\left(\dfrac{1-\underline{v}_2}{1+\underline{v}_2}\right)^2\right\}.$$ 
    
  If $\underline{v}_2\leq\underline{v}_1/(2+\underline{v}_1)$ and $\overline{v}_2\leq\overline{v}_1/(2+\overline{v}_1)$, by Propositions \ref{Thm3.1},~\ref{Thm3.2} and \ref{Thm3.3}, $$\dfrac{\overline{v}_1-\overline{v}_2-\overline{v}_1\cdot\overline{v}_2}{(1+\overline{v}_1)(\overline{v}_1-\overline{v}_2)}\leq{\rm dim}_H\left(\mathcal{L}(\psi_1)\cap\mathcal{U}(\psi_2)\right)\leq\dfrac{\underline{v}_1-\underline{v}_2-\underline{v}_1\cdot\underline{v}_2}{(1+\underline{v}_1)(\underline{v}_1-\underline{v}_2)}.$$   
\end{proof}

\section{Proofs of Theorems \ref{Uniform} and \ref{Belong}}\label{sec1}
   
   In this section, we will give the proofs of Theorems \ref{Uniform} and \ref{Belong}. 
\begin{proof}[Proof of Theorem \ref{Uniform}]
   By Lemma \ref{SET}, one has $$\mathcal{L}(\psi_1)\cap\mathcal{U}(\psi_2)\subseteq\{x\in[0,1]:\nu_\beta(x)\geq\underline{v}_1\}\cap\{x\in[0,1]:\hat{\nu}_\beta(x)\geq\underline{v}_2\}.$$ Replace the role of $\underline{v}_1$ by $v_\beta$, for $L$ large enough, we consider the set $$\mathbb{D}:=\{x\in[0,1]:v_\beta\leq\nu_\beta(x)< v_\beta+1/L\}\cap\{x\in[0,1]:\hat{\nu}_\beta(x)\geq\underline{v}_2\}.$$ By the similar discussions in Proposition \ref{Thm3.2}, if $\underline{v}_2>1$, then $\underline{v}_2\geq (v_\beta+1/L)/(1+v_\beta+1/L)$ for any $v_\beta$. Therefore, $\mathbb{D}$ is empty. Thus, $$\mathcal{U}(\psi_2)=\{x\in[0,1]:\hat{\nu}_\beta(x)=+\infty\}.$$ By Proposition \ref{Thm3.2} and Lemma \ref{infinite}, $\mathcal{U}(\psi_2)$ is countable.
   
   If $\underline{v}_2=1$, then $v_\beta=+\infty$. By {\bf Theorem SW}, ${\rm dim}_H \left(\mathbb{D}\right)=0$. If $\underline{v}_2<1$, then for $v_\beta\geq\underline{v}_2/(1-\underline{v}_2)$, one has $${\rm dim}_H \left(\mathbb{D}\right)\leq\dfrac{v_\beta-\underline{v}_2-v_\beta\cdot v_2}{(1+v_\beta)(v_\beta-\underline{v}_2)}+\dfrac{\underline{v}_2^2/L}{1-\underline{v}_2}.$$ Since $\mathcal{L}(\psi_1)\cap\mathcal{U}(\psi_2)$ is a subset of $$\cup_{N=0}^{+\infty}\cup^L_{i=1}\{x\in[0,1]:\underline{v}_1+N+(i-1)/L\leq\nu_\beta(x)<v_1+N+i/L\}\cap\mathcal{U}(\psi_2),$$ let $L$ tend to $+\infty$, one has $${\rm dim}_H\left( \mathcal{L}(\psi_1)\cap\mathcal{U}(\psi_2)\right)\leq\sup_{v_\beta\geq \underline{v}_2/(1-\underline{v}_2)}\left\{\dfrac{v_\beta-\underline{v}_2-v_\beta\cdot \underline{v}_2}{(1+v_\beta)(v_\beta-\underline{v}_2)}\right\}.$$ Regarding the right side as a function of $v_\beta$ with $v_\beta\geq\underline{v}_2/(1-\underline{v}_2)$, we obtain the maximum at $v_\beta=2\underline{v}_2/(1-\underline{v}_2)$. Therefore, $${\rm dim}_H \left(\mathcal{U}(\psi_2)\right)\leq\left(\dfrac{1-\underline{v}_2}{1+\underline{v}_2}\right)^2.$$ Combining with Proposition \ref{Thm3.1}, one has $$0\leq{\rm dim}_H \left(\mathcal{U}(\psi_2)\right)\leq\min\left\{\dfrac{1}{1+\overline{v}_2},~\left(\dfrac{1-\underline{v}_2}{1+\underline{v}_2}\right)^2\right\}.$$
   
   If $\overline{v}_2\leq1$, then $\underline{v}_2\leq\overline{v}_2\leq1$, one also has $${\rm dim}_H \left(\mathcal{U}(\psi_2)\right)\leq\min\left\{\dfrac{1}{1+\overline{v}_2},~\left(\dfrac{1-\underline{v}_2}{1+\underline{v}_2}\right)^2\right\}.$$ To obtain the lower bound of the Hausdor dimension of $\mathcal{L}(\psi_1)\cap\mathcal{U}(\psi_2)$, we will construct a Cantor type subset $E$ of $\mathcal{L}(\psi_1)\cap\mathcal{U}(\psi_2)$. By Lemma \ref{SET}, $$\{x\in[0,1]:\nu_\beta(x)>\overline{v}_1\}\cap\{x\in[0,1]:\hat{\nu}_\beta(x)>\overline{v}_2\}\subseteq\mathcal{L}(\psi_1)\cap\mathcal{U}(\psi_2).$$ We replace the role of $\overline{v}_1$ by $v_\beta$ with $v_\beta\geq\overline{v}_2/(1-\overline{v}_2)$. Fix $\delta>0$, consider $$\{x\in[0,1]:\nu_\beta(x)= v_\beta+\delta\}\cap\{x\in[0,1]:\hat{\nu}_\beta(x)=\overline{v}_2+\delta\}.$$If $\overline{v}_2=0$, then let $$n'_k=k^k,\quad m'_k=\lfloor(1+ v_\beta+\delta)n'_k\rfloor,\quad {\rm for \ }k=1,2,\cdots$$ If $\overline{v}_2>0$, then, let $$n'_k=\left\lfloor\left(\dfrac{v_\beta+\delta}{\overline{v}_2+\delta}\right)^k\right\rfloor,\quad m'_k=\lfloor(1+v_\beta+\delta)n'_k\rfloor,\quad k=1,2,\cdots$$ Arguing as in the proof of Proposition \ref{Thm3.3}, one has $${\rm dim}_H\left(\mathcal{L}(\psi_1)\cap\mathcal{U}(\psi_2)\right)\geq \dfrac{v_\beta-\overline{v}_2-v_\beta\overline{v}_2}{(1+v_\beta)(v_\beta-\overline{v}_2)}.$$ Regarding the right side as a function of $v_\beta$ and taking $v_\beta\geq\overline{v}_2/(1-\overline{v}_2)$ into account, we obtain the maximum at $v_\beta=2\overline{v}_2/(1-\overline{v}_2)$. Then, $${\rm dim}_H\left(\mathcal{U}(\psi_2)\right)\geq\left(\dfrac{1-\overline{v}_2}{1+\overline{v}_2}\right)^2.$$  
\end{proof}   
      
\begin{proof}[Proof of Theorem \ref{Belong}]
   If $\underline{v}_1=\overline{v}_1=0$, then $$\lim_{n\rightarrow\infty}\dfrac{-\log_\beta\psi_1(n)}{n}=0.$$ If $\underline{v}_2>0$, by the definitions of $\underline{v}_2$ and $\underline{v}_1=\overline{v}_1=0$, for any $\varepsilon\in(0,\underline{v}_2/2)$, there is an integer $n_0$ such that for any $n\geq n_0$, one has $$\psi_2(n)\leq \beta^{-n(\underline{v}_2-\varepsilon)}<\beta^{-n\varepsilon}\leq \psi_1(n).$$
   
   For any $x\in\mathcal{U}(\psi_2)$, by the same argument as Proposition \ref{Pro}, we have $$\mathcal{U}(\psi_2)\subseteq\mathcal{L}(\psi_1).$$  
\end{proof}

\section{Examples}

   In this section, we will show that the upper and lower bounds of Theorems \ref{Nonspecial} and \ref{Uniform} can be all reached. Examples \ref{Exa1}, \ref{Exa2} and \ref{Exa3} explain that the upper bound estimation $\dfrac{1}{1+\overline{v}_2}$, $\left(\dfrac{1-\underline{v}_2}{1+\underline{v}_2}\right)^2$ and $\dfrac{\underline{v}_1-\underline{v}_2-\underline{v}_1\underline{v}_2}{(1+\underline{v}_1)(\underline{v}_1-\underline{v}_2)}$ are reachable, respectively.
\begin{Example}\label{Exa1}
   Let $\psi_1(n)=1$, for $n=1,2,\cdots$ and
\begin{equation*}
   \psi_2(n)=
   \begin{cases}
   \beta^{-3n}, &\text{ if $n=k^k$}\\
   1, &\text{ if $n\neq k^k$}
   \end{cases},\quad  k=1,2,\cdots
\end{equation*} 
   Then, $\underline{v}_1=\overline{v}_1=0$, $\underline{v}_2=0$, $\overline{v}_2=3$ and $${\rm dim}_H\left(\mathcal{L}(\psi_1)\cap\mathcal{U}(\psi_2)\right)=\dfrac{1}{4}=\dfrac{1}{1+\overline{v}_2}.$$	
\end{Example}    
  
\begin{proof}
   By Proposition \ref{Thm3.1}, we only need to show that $${\rm dim}_H\left(\mathcal{L}(\psi_1)\cap\mathcal{U}(\psi_2)\right)\geq1/4.$$ Now, we construct a Cantor subset $E$ of $\mathcal{L}(\psi_1)\bigcap\mathcal{U}(\psi_2)$ such that $${\rm dim}_H(E)\geq 1/4.$$ Let $n_k=k^k$ and $m_k=4k^k$, for $k=1,2,\cdots$, we construct a Cantor subset $E$ by the same way as Proposition \ref{Thm3.3}, then $${\rm dim}_H\left(\mathcal{L}(\psi_1)\cap\mathcal{U}(\psi_2)\right)\geq1/4=\dfrac{1}{1+\overline{v}_2}.$$ 	
\end{proof}

\begin{Example}\label{Exa2}
   Let $\psi_1(n)=1$, for $n=1,2,\cdots$ and
\begin{equation*}
   \psi_2(n)=
   \begin{cases}
   \beta^{-2n}, &\text{if $n=4^k$}\\
   \beta^{-n/2}, &\text{if $n\neq 4^k$}
   \end{cases},\quad k=1,2,\cdots
\end{equation*} 
   Then, $\underline{v}_1=\overline{v}=0$, $\underline{v}_2=1/2$, $\overline{v}_2=2$ and $${\rm dim}_H\left(\mathcal{L}(\psi_1)\cap\mathcal{U}(\psi_2)\right)=\dfrac{1}{9}=\left(\dfrac{1-\underline{v}_2}{1+\underline{v}_2}\right)^2.$$
\end{Example}   

\begin{proof}
   Since $\underline{v}_1/(2+\underline{v}_1)<\underline{v}_2$, by Proposition \ref{Thm3.2}, the proof is completed by showing $${\rm dim}_H\left(\mathcal{L}(\psi_1)\cap\mathcal{U}(\psi_2)\right)\geq\dfrac{1}{9}=\left(\dfrac{1-\underline{v}_2}{1+\underline{v}_2}\right)^2.$$ Let $n_k=4^k$ and $m_k=3\cdot 4^k$, for $k=1,2,\cdots$, we construct a Cantor subset $E$ by the same way as Proposition \ref{Thm3.3}, then $${\rm dim}_H\left(\mathcal{L}(\psi_1)\cap\mathcal{U}(\psi_2)\right)\geq\dfrac{1}{9}=\left(\dfrac{1-\underline{v}_2}{1+\underline{v}_2}\right)^2.$$     	
\end{proof}  
  
\begin{Example}\label{Exa3}
   For $k=1,2,\cdots$, let
\begin{equation*}
   \psi_1(n)=
   \begin{cases}
   \beta^{-n/2}, & \text{ if $n=3^k$}\\
   \beta^{-n}, & \text{ if $n\neq 3^k$}
   \end{cases}, \ \
   \psi_2(n)=
   \begin{cases}
   \beta^{-n/2}, &\text{ if $n=3^k$}\\
   \beta^{-n/6,} &\text{ if $n\neq 3^k$}.
   \end{cases}
\end{equation*} 
   Then, $\underline{v}_1=1/2$, $\overline{v}_1=1$, $\underline{v}_2=1/6$, $\overline{v}_2=1/2$ and $${\rm dim}_H\left(\mathcal{L}(\psi_1)\cap\mathcal{U}(\psi_2)\right)=\dfrac{1}{2}=\dfrac{\underline{v}_1-\underline{v}_2-\underline{v}_1\underline{v}_2}{(1+\underline{v}_1)(\underline{v}_1-\underline{v}_2)}.$$
\end{Example}

\begin{proof}
   Since $\underline{v}_2\leq\underline{v}_1/(2+\underline{v}_1)$, by Proposition \ref{Thm3.2}, one has $${\rm dim}_H\left(\mathcal{L}(\psi_1)\cap\mathcal{U}(\psi_2)\right)\leq \dfrac{\underline{v}_1-\underline{v}_2-v_1\cdot\underline{v}_2}{(1+\underline{v}_1)(\underline{v}_1-\underline{v}_2)}=\dfrac{1}{2}.$$ It suffices to show ${\rm dim}_H\left(\mathcal{L}(\psi_1)\cap\mathcal{U}(\psi_2)\right)\geq1/2$. Let $n_k=4^k$ and $m_k=3\cdot 4^k$, for $k=1,2,\cdots$, we construct a Cantor subset $E$ by the same way as Proposition \ref{Thm3.3}, then $${\rm dim}_H\left(\mathcal{L}(\psi_1)\cap\mathcal{U}(\psi_2)\right)\geq\dfrac{1}{2}=\dfrac{\underline{v}_1-\underline{v}_2-\underline{v}_1\underline{v}_2}{(1+\underline{v}_1)(\underline{v}_1-\underline{v}_2)}.$$          
\end{proof}

   Examples \ref{Exa4} and \ref{Exa5} explain that the lower bound estimation $0$ and $\left(\dfrac{1-\overline{v}_2}{1+\overline{v}_2}\right)^2$ are reachable, respectively.
\begin{Example}\label{Exa4}
   Let $\psi_1(n)=1$ for $n=1,2,\cdots$ and
\begin{equation*}
   \psi_2(n)=
   \begin{cases}
   \beta^{-3n}, &\text{ if $n=4^k$}\\
   \beta^{-n}, &\text{ if $n\neq 4^k$}
   \end{cases}, \quad k=1,2,\cdots
\end{equation*} 
   Then, $\underline{v}_1=\overline{v}_1=0$, $\underline{v}_2=1$, $\overline{v}_2=3$ and $${\rm dim}_H\left(\mathcal{L}(\psi_1)\cap\mathcal{U}(\psi_2)\right)=0.$$	
\end{Example}  
  
\begin{proof}
   It remains to prove ${\rm dim}_H\left(\mathcal{L}(\psi_1)\cap\mathcal{U}(\psi_2)\right)\leq0$. In fact, by Proposition \ref{Thm3.2}, one has $${\rm dim}_H\left(\mathcal{L}(\psi_1)\cap\mathcal{U}(\psi_2)\right)\leq\left(\dfrac{1-\underline{v}_2}{1+\underline{v}_2}\right)^2=\left(\dfrac{1-1}{1+1}\right)^2=0.$$ 
\end{proof}
  
\begin{Example}\label{Exa5}For $k=1,2,\cdots$, let
\begin{equation*}
\psi_1(n)=
\begin{cases}
\beta^{-3n}, & \text{ if $n=2k+1$}\\
\beta^{-10n/3}, & \text{ if $n=2k$}
\end{cases}, \ \
\psi_2(n)=
\begin{cases}
\beta^{-21n/32}, &\text{ if $n=2k+1$}\\
\beta^{-2n/3}, &\text{ if $n=2k$}. 
\end{cases}
\end{equation*} Then, $\underline{v}_1=3$, $\overline{v}_1=10/3$, $\underline{v}_2=21/32$ and $\overline{v}_2=2/3$. One has $${\rm dim}_H\left(\mathcal{L}(\psi_1)\cap\mathcal{U}(\psi_2)\right)=\dfrac{1}{25}=\left(\dfrac{1-\overline{v}_2}{1+\overline{v}_2}\right)^2.$$ 	
\end{Example}  
  
\begin{proof}
   Since $\underline{v}_1/(2+\underline{v}_1)<\overline{v}_2$, by Proposition \ref{Thm3.3}, the proof is completed by showing $${\rm dim}_H\left(\mathcal{L}(\psi_1)\cap\mathcal{U}(\psi_2)\right)\leq\dfrac{1}{25}=\left(\dfrac{1-\overline{v}_2}{1+\overline{v}_2}\right)^2.$$ In fact, since $\underline{v}_2\leq\underline{v}_1/(2+\underline{v}_1)$, by Proposition \ref{Thm3.2}, we have $${\rm dim}_H\left(\mathcal{L}(\psi_1)\cap\mathcal{U}(\psi_2)\right)\leq\dfrac{1}{25}=\left(\dfrac{1-\overline{v}_2}{1+\overline{v}_2}\right)^2.$$
\end{proof}  
  
   The following two examples explain that the lower bound estimation $\dfrac{\overline{v}_1-\overline{v}_2-\overline{v}_1\overline{v}_2}{(1+\overline{v}_1)(\overline{v}_1-\overline{v}_2)}$ is reachable. We consider the cases of $\underline{v}_1=\overline{v}_1$ and $\underline{v}_2=\overline{v}_2$, respectively.   
\begin{Example}\label{Exa6}
   Let $\psi_1(n)=\beta^{-n}$ for $n=1,2,\cdots$ and
   \begin{equation*}
   \psi_2(n)=
   \begin{cases}
   1, &\text{ if $n=2k+1$}\\
   \beta^{-n/4}, &\text{ if $n=2k$}
   \end{cases},~k=1,2,\cdots
   \end{equation*} 
   Then, $\underline{v}_1=\overline{v}_1=1$, $\underline{v}_2=0$, $\overline{v}_2=1/4$ and $${\rm dim}_H\left(\mathcal{L}(\psi_1)\cap\mathcal{U}(\psi_2)\right)=\dfrac{1}{3}=\dfrac{\overline{v}_1-\overline{v}_2-\overline{v}_1\overline{v}_2}{(1+\overline{v}_1)(\overline{v}_1-\overline{v}_2)}.$$
\end{Example}  
  
\begin{proof}
   Since $\overline{v}_2\leq\overline{v}_1/(2+\overline{v}_1)$, by Proposition \ref{Thm3.3}, what is left is to show $${\rm dim}_H\left(\mathcal{L}(\psi_1)\cap\mathcal{U}(\psi_2)\right)\leq\dfrac{1}{3}=\dfrac{\overline{v}_1-\overline{v}_2-\overline{v}_1\overline{v}_2}{(1+\overline{v}_1)(\overline{v}_1-\overline{v}_2)}.$$ For any $x\in[0,1]$, denote its $\beta$-expansion by $$x=\dfrac{a_1}{\beta}+\dfrac{a_2}{\beta^2}+\cdots+\dfrac{a_n}{\beta^n}+\cdots,$$ where $a_i\in\{0,\cdots,\lceil\beta\rfloor\}$, for all $i\geq1$. Let $$a_{n'_i}>0,\quad a_{n'_i+1}=\cdots=a_{m'_i-1}=0,\quad a_{m'_i}>0.$$ If $x\in\mathcal{L}(\psi_1)\cap\mathcal{U}(\psi_2)$, since $v_1>0$, then one has $$\limsup_{i\rightarrow\infty}(m'_i-n'_i)=+\infty.$$ Arguing as in the proof of Proposotion \ref{Thm3.2}, we take the maximal subsequences $\{n_k:k\geq1\}$ and $\{m_k:k\geq1\}$ of $\left\{n'_i:i\geq 1\right\}$ and $\left\{m'_i:i\geq1\right\}$, respectively, such that the sequence $\{m_k-n_k: k\geq1\}$ is non-decreasing. Notice $\beta^{n_k-m_k}<T^{n_k}_\beta x<\beta^{n_k-m_k+1}$, one has $$\limsup_{k\rightarrow\infty}\dfrac{m_k-n_k}{n_k}=1.$$ Since $x\in\mathcal{U}(\psi_2)$, there is an integer $k_0$ such that for any $k\geq k_0$, one has $m_k-n_k\geq\lfloor n_{k+1}/4\rfloor$. If not, for any $j\geq1$, there is an integer $k_j$ such that $m_{k_j}-n_{k_j}< n_{k_{j+1}}/4$. Since one of $n_{k_{j+1}}$ and $n_{k_{j+1}}+1$ is even, denote it by $l_{k_{j+1}}$, for any integer $n\in[1,l_{k_{j+1}}]$, one has $$T^n_\beta x>\beta^{n_{k_j}-m_{k_j}}>\beta^{-l_{k_{j+1}}/4}=\psi_2(l_{k_{j+1}}).$$ It contradicts $x\in\mathcal{U}(\psi_2)$.
   
   Choose the subsequence $\{n_{k_i}:i\geq1\}$ and $\{m_{k_i}:i\geq1\}$ of $\{n_k:k\geq1\}$ and $\{m_k:k\geq1\}$, respectively, such that $$\lim_{i\rightarrow\infty}\dfrac{m_{k_i}-n_{k_i}}{n_{k_i}}=1.$$ For simplicity, let $\{n_k:k\geq1\}$ and $\{m_k:k\geq1\}$ stand for $\{n_{k_i}:i\geq1\}$ and $\{m_{k_i}:i\geq1\}$, respectively. For any $\varepsilon>0$, there is an integer $k'$ such that for any $k\geq k'$, one has $$(1-\varepsilon)n_k\leq m_k-n_k\leq(1+\varepsilon)n_k,\quad m_k-n_k\geq\dfrac{n_{k+1}}{4}-2.$$ Arguing as in the proof of Proposition \ref{Thm3.2}, one has $${\rm dim}_H\left(\mathcal{L}(\psi_1)\cap\mathcal{U}(\psi_2)\right)\leq\dfrac{1}{3}=\dfrac{\overline{v}_1-\overline{v}_2-\overline{v}_1\overline{v}_2}{(1+\overline{v}_1)(\overline{v}_1-\overline{v}_2)}.$$ 
\end{proof} 

\begin{Example}\label{Exa7}
   Let 
\begin{equation*}
   \psi_1(n)=
   \begin{cases}
   \beta^{-n/3}, &\text{ if $n=2k+1$}\\
   \beta^{-2n/3}, &\text{if $n=2k$}
   \end{cases},~k=1,2,\cdots
\end{equation*} 
   and $\psi_2(n)=\beta^{-2n/11}$ for $n=1,2,\cdots$. Then, $\underline{v}_1=1/3$, $\overline{v}_1=2/3$, $\underline{v}_2=\overline{v}_2=2/11$ and $${\rm dim}_H\left(\mathcal{L}(\psi_1)\cap\mathcal{U}(\psi_2)\right)=\dfrac{9}{20}=\dfrac{\overline{v}_1-\overline{v}_2-\overline{v}_1\overline{v}_2}{(1+\overline{v}_1)(\overline{v}_1-\overline{v}_2)}.$$
\end{Example}

\begin{proof}
   Since $\overline{v}_2\leq\overline{v}_1/(2+\overline{v}_1)$, by Proposition \ref{Thm3.3}, what is left is to show $${\rm dim}_H\left(\mathcal{L}(\psi_1)\cap\mathcal{U}(\psi_2)\right)\leq\dfrac{9}{20}=\dfrac{\overline{v}_1-\overline{v}_2-\overline{v}_1\overline{v}_2}{(1+\overline{v}_1)(\overline{v}_1-\overline{v}_2)}.$$ In fact, since $\underline{v}_2=\dfrac{2}{11}\leq\dfrac{1}{7}=\underline{v}_1/(2+\underline{v}_1)$, by Proposition \ref{Thm3.2}, $${\rm dim}_H\left(\mathcal{L}(\psi_1)\cap\mathcal{U}(\psi_2)\right)\leq\dfrac{9}{20}=\dfrac{\overline{v}_1-\overline{v}_2-\overline{v}_1\overline{v}_2}{(1+\overline{v}_1)(\overline{v}_1-\overline{v}_2)}.$$ 
\end{proof}

\addcontentsline{toc}{section}{Reference}
\bibliographystyle{amsplain}

\end{document}